\documentclass[11pt,oneside]{amsart}

\usepackage{bbding}
\usepackage{hyperref}
\usepackage{tikz} 
\usepackage{amssymb}
\usepackage{amscd}
\usepackage{amssymb, amsthm}
\usepackage{enumerate, amsfonts, latexsym,epsfig, color, lpic}
\usepackage{graphicx}
\usepackage{pictexwd,dcpic}
\usepackage{upgreek}
\usepackage{mathptmx}
\usepackage{palatino}
\usepackage{thm-restate}
\usepackage{epstopdf}

\input xy
\xyoption{all}

\usepackage[notref,notcite]{showkeys} 

\newtheorem {theorem}{Theorem} [section]
\newtheorem {lemma} [theorem] {Lemma}
\newtheorem {proposition} [theorem] {Proposition}

\newtheorem {corollary} [theorem] {Corollary}
\newtheorem {definition} [theorem] {Definition}

\newtheorem {remark} [theorem] {Remark}

\newtheorem{assumption} [theorem] {Assumption}

\def\N {\mathbb N}

\def\mc {\mathcal}

\def\walmost {$\omega$--almost }

\def\w {$\omega$--}

\begin{document}

\title[Discrete representations into ${\rm PSL}(2,\mathbb{R})$]
{\bf Discrete representations of finitely generated groups into ${\rm PSL}(2,\mathbb{R})$}

\author{Hao Liang}
\address{Hao Liang, School of Mathematics, Sun Yat-sen University, Guangzhou, 
Guangdong, 510275, People's Republic
of China} \email{liangh63@mail.sysu.edu.cn; lianghao1019@hotmail.com}



\begin{abstract}
We prove a factorization theorem for Fuchsian groups similar to those proved by Agol and Liu for 3-manifold groups.  As an application, we build Makanin-Razborov diagrams, which parametrize the collection of all discrete representations from an arbitrary but fixed finitely generated group $G$ to ${\rm PSL}(2, \mathbb{R})$. We define a new class of groups called ${\rm PSL}(2, \mathbb{R})$-discrete limit groups and then use the factorization theorem to obtain useful information about this class of groups.  
\end{abstract}

\thanks{The author is supported by NSFC (No.11701581 and No.11521101).}

\maketitle

\setcounter{tocdepth}{1}
\tableofcontents

\vspace{4mm}

\section{\large Introduction}


Discrete and faithful representations into ${\rm PSL}(2, \mathbb{R})$ are fundamental objects in many important branches of mathematics and are heavily studied. In this paper, we consider all discrete representations into ${\rm PSL}(2, \mathbb{R})$ without requiring them to be faithful. We use tools and ideals from the theory of limit groups and group actions on $\mathbb{R}$-trees, developed by Sela, Rips and others. Here is the simplified version of one of our main results (Theorem \ref{finite factoring set}):

\begin{theorem}\label{TMR}
Let $G$ be a finitely generated group. There exists a finite collection $\{\Gamma^1, \dots, \Gamma^n\}$, where each $\Gamma^i$ is an amalgamated product of a hyperbolic 2-orbifold group with finitely many virtually cyclic abelian groups, such that any discrete representation $f$ of $G$ into ${\rm PSL}(2, \mathbb{R})$ factors through $\Gamma^i$ for some $1\leq i\leq n$. 
\end{theorem}

Theorem \ref{TMR} allows one to use the theory of Makanin-Razborov diagrams to encode most of the ``unfaithfulness'' of ${\rm Hom}_d(G, {\rm PSL}(2, \mathbb{R}))$, the collection of all discrete representations of $G$ into ${\rm PSL}(2, \mathbb{R})$. Theorem \ref{TMR}, in a sense, decomposes ${\rm Hom}_d(G, {\rm PSL}(2, \mathbb{R}))$ into $\{{\rm Hom}_d(\Gamma_i, {\rm PSL}(2, \mathbb{R}))\mid i=1, \dots, n\}$ and $\{{\rm Hom}(G, \Gamma_i)\mid i=1, \dots, n\}$. By the extra information about $\Gamma_i$ and the factoring maps provided in Theorem \ref{finite factoring set}, in the above decomposition one only needs to consider a subset of ${\rm Hom}_d(\Gamma_i, {\rm PSL}(2, \mathbb{R}))$, whose complexity is mostly captured by sets of the form ${\rm Hom}^f_d(\pi_1(\mc{O}), {\rm PSL}(2, \mathbb{R}))$. Here ${\rm Hom}^f_d(\pi_1(\mc{O}), {\rm PSL}(2, \mathbb{R}))$ denote the collection of all discrete faithful representations from the fundamental group of a hyperbolic two orbifold $\mc{O}$ to ${\rm PSL}(2, \mathbb{R})$. Hence the ``unfaithfulness'' of ${\rm Hom}_d(G, {\rm PSL}(2, \mathbb{R}))$ is mostly captured by the sets $\{{\rm Hom}(G, \Gamma_i)\mid i=1, \dots, n\}$. 

To understand ${\rm Hom}(G, \Gamma^i)$, one can apply the theory of Makanin-Razborov diagrams. The theory of Makanin-Razborov diagrams is introduced by Sela in \cite{Sela2} to study ${\rm Hom}(G, \mathbb{F})$, where $\mathbb{F}$ is a nonabelian free group. This is the first step in Sela's solution of Tarski problems regarding the elementary theory of free groups. (Kharlampovich and Myasnikov \cite{KM} have another approach to these problems). Sela's work on Makanin-Razborov diagrams has been generalized to several important classes of finitely generated groups(\cite{Ali}, \cite{CK}, \cite{G2}, \cite{JS}, \cite{R}, \cite{Sela2}, \cite{Sela3}), among which is Gromov hyperbolic groups.  Makanin-Razborov diagrams for a finitely generated group $\Gamma$ are finite diagrams of homomorphisms. They give  parametrizations of sets of the form ${\rm Hom}(G, \Gamma)$, where $G$ is an arbitrary fixed finitely generated group. We want to emphasize that finiteness is the crucial property of Makanin-Razborov diagrams, whose existence is trivial without the finiteness requirement. In our current setting, we will see in Section \ref{application} that each $\Gamma^i$ in Theorem \ref{TMR} is a Gromov hyperbolic group. As a result, each ${\rm Hom}(G, \Gamma^i)$ can be understood by constructing the Makanin-Razborov diagrams for $\Gamma^i$. (See \cite{Sela3} and \cite{R}).    

Since there are finitely many $\Gamma^i$, the union of the Makanin-Razborov diagrams for each ${\rm Hom}(G, \Gamma^i)$ is a finite diagram of homomorphisms. This diagram, together with some well understood homomorphisms between hyperbolic 2-orbifold groups, parametrizes ${\rm Hom}_d(G, {\rm PSL}(2, \mathbb{R}))$ in a similar way the Makanin-Razborov diagrams for $\Gamma^i$ parametrize ${\rm Hom}(G, \Gamma^i)$. We call this finite diagram of homomorphisms a {\em discrete Makanin-Razborov diagram} for ${\rm PSL}(2, \mathbb{R})$.  See Section \ref{application} for the precise description of this diagram. 

$\Gamma$-limit groups naturally arise in the construction of Makanin-Razborov diagrams for $\Gamma$ and they are natural objects to study. We introduce ${\rm PSL}(2, \mathbb{R})$-discrete limit groups (See Definition \ref{omega kernel}), which naturally arise in the construction of discrete Makanin-Razborov diagrams for ${\rm PSL}(2, \mathbb{R})$. As an application of Theorem \ref{TMR}, we obtain useful information about ${\rm PSL}(2, \mathbb{R})$-discrete limit groups. 

\begin{theorem}\label{finite pre finite abe}
${\rm PSL}(2, \mathbb{R})$-discrete limit groups are finitely presented. All abelian subgroups of ${\rm PSL}(2, \mathbb{R})$-discrete limit groups are finitely generated. 
\end{theorem}

It is not hard to reduce the proof of Theorem \ref{TMR} to proving a factorization theorem for hyperbolic 2-orbifold groups (Theorem \ref{MT}), whose simplified version is the following: 

\begin{theorem}\label{t: main factor one conept}
Let $G$ be a finitely generated group. Suppose $\mc{O}$ is an orientable hyperbolic cone type 2-orbifold of finite type. Then there exists a number $N$, depending only on $G$ and $\mc{O}$, such that the following is true:  Denote by $\mc{O}_p$ the 2-orbifold obtained from $\mc{O}$ by replacing a cone point of order greater than $N$ (assuming $\mc{O}$ has such a cone point) with a puncture.  Then every homomorphism $f: G\rightarrow \pi_1(\mc{O})$ factors through a group of the form $\pi_1(\mc{O}_p)*_{\mathbb{Z}}(\mathbb{Z}\oplus\mathbb{Z}/a\mathbb{Z})$. Moreover, the factoring map $\pi_f:\pi_1(\mc{O}_p)*_\mathbb{Z}(\mathbb{Z}\oplus\mathbb{Z}/a\mathbb{Z}) \rightarrow \pi_1(\mc{O})$ is an extension of the natural quotient map $\pi: \pi_1(\mc{O}_p)\rightarrow \pi_1(\mc{O})$.
\end{theorem}

In the general version of the above theorem, replacing multiple cone points with punctures are allowed. See Theorem \ref{MT} for the exact statement. 

Factorization theorems of the same flavor have been proved by Agol and Liu in \cite{Agol-liu} for maps from a finitely presented group $G$ to the fundamental group of an orientable aspherical compact 3-manifold $M$. (As an application of their factorization theorems, Liu proves a factorization theorem similar to Theorem \ref{TMR} for torsion free Kleinian groups of uniformly bounded covolume in \cite{liu}.) In their theorems, instead of a cone point with sufficiently large order being replaced by a puncture, a sufficiently short simple closed geodesic in a hyperbolic piece (or an exceptional fibre at a sufficiently sharp cone point in a Seifert fibered piece) is drilled out. The group $\mathbb{Z}\oplus\mathbb{Z}/a\mathbb{Z}$ ``attached'' to the hyperbolic 2-orbifold group in Theorem \ref{t: main factor one conept} can be thought of an algebraic analog of the Dehn extension introduced in \cite{Agol-liu}. Agol and Liu's proof is topological and does not seem to generalize easily to the case where $G$ is finitely generated. (However, when the 3-manifold $M$ is hyperbolic, Liu did generalized their theorem to finitely generated case in \cite{liu}. ) Unlike their proof, our proof of the factorization theorem relies on the theory of group actions on $\mathbb{R}$-trees and Sela's theory of limit groups. We believe that the technic in this paper and the theory developed in our paper \cite{GHL} will allow us to generalize Agol-Liu's factorization theorem to the finitely generated case. 

Agol and Liu's factorization theorems serve as the crucial step in their proof of J. Simon's conjecture about epimorphisms between knot groups. Like Agol and Liu's factorization theorem, our factorization theorem (Theorem \ref{MT2}) also has an application in the study of homomorphisms between fundamental groups of 3-manifolds: 

In \cite{GHL}, we prove that the collection of all compact 3-manifold groups are equationally Noetherian, which answers a question of Agol and Liu. This result implies that sequences of proper epimorphisms between compact 3-manifold groups have finite lengths, which answers a question of Reid, Wang and Zhou \cite{RWZ}. 

An important step in the proof of the above result is to show that limit groups over fundamental groups of Seifert fibered 3-manifolds with hyperbolic base orbifolds are finitely presented and that all their abelian subgroups are finitely generated. The proofs of these facts rely on Theorem \ref{finite pre finite abe}, which can be considered as a corollary of Theorem \ref{MT2}.

We now give a brief outline of this paper. In Section 2, we set up the basic notations and state the stronger version of Theorem \ref{t: main factor one conept} (Theorem \ref{MT2}), which is the key in proving the other main theorems. The main idea of the proof of Theorem \ref{MT2} will also be explained in Section 2.  The proof of Theorem \ref{MT2} is separated into eight sections: Section \ref{s:Bounded} to Section \ref{proof}. We will give more information about what each of these sections is about after we explain the main idea of the proof in Section 2. In the last section, we prove the other main theorems as applications of Theorem \ref{MT2}. 

I am deeply grateful to Daniel Groves for telling me about the problem, its connection with the theory of limit groups and the construction in Section \ref{es}, which is the key tool in the proof of Theorem \ref{MT2}. This paper benefits from many of the discussions Daniel Groves, Michael Hull and I had when we worked on a closely related project (\cite{GHL}). I want to thank them for that. I thank Michael Siler for pointing out a mistake in the first version of the statement of Theorem \ref{MT2}. I also want to thank Lars Louder for an interesting conversation related to this work.

\section{\large Statement of the factorization theorems}\label{s:sandp}

Let $\mc{O}$ be an orientable hyperbolic cone type $2$-orbifold of finite type, i.e. the only singularities of $\mc{O}$ are cone points and the underlying surface of $\mc{O}$ has finite genus and finitely many punctures and cone points. Equivalently, $\mc{O}$ is the $2$-orbifold corresponding to a finitely generated Fuchsian group, i.e. discrete subgroup of ${\rm PSL}(2, \mathbb{R})$. Let $\bar p=(p_1,\dots, p_k)$ be a tuple of $k$ punctures of $\mc{O}$. Let $P_j$ be the subgroup of $\pi_1(\mc{O})$ generated by the simple loop around $p_j$. Note that $P_j$ is only well defined up to conjugation.

\begin{definition}\label{orbifolddef}
Given $\mc{O}$ as above. Let $\bar n=(n_1, \dots, n_k)\in \mathbb{N}^k$.  Denote by $\mc{O}_{\bar n}$ the 2-orbifold obtained from $\mc{O}$ by replacing the punctures $p_j$ with a cone point $c_j$ of order $n_j$. Let $\pi_{\bar n}:\pi_1(\mc{O})\rightarrow \pi_1(\mc{O}_{\bar n})$ be the natural projection. 
\end{definition}

\begin{remark}
Note that when $n_j$ is big enough for all $1\leq j\leq k$, $\mc{O}_{\bar n}$ also has a hyperbolic structure. This is the only case we consider in this paper. 
\end{remark}

Let $C_j$ be the subgroup of $\pi_1(\mc{O}_{\bar n})$ generated by the simple loop around $c_j$ for $1\leq j\leq k$.  
Note that $C_j$ is only well defined up to conjugation. 


\begin{definition}\label{amalg}
Given $\mc{O}$ as above. Let $\bar A=(A_1, \dots, A_k)$ be a tuple of groups. Define a graph of groups as follow: 
\begin{enumerate}

\item The underlying graph is the complete bipartite graph $K_{1, k}$. 

\item The  center vertex groups is $\pi_1(\mc{O})$ and the other $k$ vertex groups are $A_1, \dots, A_k$. 
\item  The edge groups are cyclic.

\item The boundary homomorphism takes the edge group connecting $A_j$ and $\pi_1(\mc{O})$ isomorphically to $P_j$. 

\end{enumerate}
We call the fundamental group the above graph of group {\em the amalgamated product of $\pi_1(\mc{O})$ with  $\bar A$ along $\bar p$} and denote it by $\pi_1(\mc{O})*_{\bar p}\bar A$.
\end{definition}

With the above notation, the main theorem is as follow. 

\begin{theorem}\label{MT}
Let $G$ be a finitely generated group, $\mc{O}$ be an orientable hyperbolic cone type 2-orbifold of finite type and $\bar p=(p_1, \dots, p_k)$ be a tuple of punctures of $\mc{O}$. Suppose $\mc{O}_{\bar n}$ has hyperbolic structure for some $\bar n\in \mathbb{N}^k$. Then there exists $\bar N\in \mathbb{N}^k$, depending on $G$ and $\mc{O}$, with the following property: For any $\bar n\geq \bar N$, every homomorphism $f:G\rightarrow\pi_1(\mc{O}_{\bar n})$ factors through some $\pi_1(\mc{O})*_{\bar p}\bar A$, where $\bar A=(\mathbb{Z}\oplus(\mathbb{Z}/a_1\mathbb{Z}), \dots, \mathbb{Z}\oplus(\mathbb{Z}/a_k\mathbb{Z}))$. Moreover the factoring map $\pi_f: \pi_1(\mc{O})*_{\bar p}\bar A\rightarrow \pi_1(\mc{O}_{\bar n})$ is an extension of $\pi_{\bar n}$.  
\end{theorem}

It is easy to see that Theorem \ref{MT} follows from the next theorem.

\begin{theorem}\label{MT2}
Let $G$ be a finitely generated group, $\mc{O}$ be an orientable hyperbolic cone type 2-orbifold of finte type and $\bar p=(p_1, \dots, p_k)$ be a tuple of punctures of $\mc{O}$. Suppose $\mc{O}_{\bar n}$ also has hyperbolic structure for some $\bar n\in \mathbb{N}^k$. Suppose $\bar n_i\in \mathbb{N}^k\rightarrow \infty$ as $i\rightarrow \infty$.  Let $\{f_i: G\rightarrow \pi_1(\mc{O}_{\bar n_i})\}$ be a sequence of homomorphisms. Then there exists $\bar A=(\mathbb{Z}\oplus(\mathbb{Z}/a_1\mathbb{Z}), \dots, \mathbb{Z}\oplus(\mathbb{Z}/a_k\mathbb{Z}))$ and $\pi_1(\mc{O})*_{\bar p}\bar A$ such that $f_i$ factor through $\pi_1(\mc{O})*_{\bar p}\bar A$ for infinitely many $i$.  Moreover, each of the factoring map $\pi_{f_i}: \pi_1(\mc{O})*_{\bar p}\bar A\rightarrow \pi_1(\mc{O}_{\bar n_i})$ is an extension of the natural map $\pi_{\bar n_i}$. 
\end{theorem}

\begin{remark}
In Theorem \ref{MT}, $\bar A=(\mathbb{Z}\oplus(\mathbb{Z}/a_1\mathbb{Z}), \dots, \mathbb{Z}\oplus(\mathbb{Z}/a_k\mathbb{Z}))$ and $\pi_1(\mc{O})*_{\bar p}\bar A$ depend on $f$. However, in Theorem \ref{MT2}, $\bar A=(\mathbb{Z}\oplus(\mathbb{Z}/a_1\mathbb{Z}), \dots, \mathbb{Z}\oplus(\mathbb{Z}/a_k\mathbb{Z}))$ and $\pi_1(\mc{O})*_{\bar p}\bar A$ depend only on the sequence $\{f_i\}$, but not on $i$. 
\end{remark}

We make the following assumption in Section \ref{s:Bounded} through Section \ref{proof}, which are devoted to the proof of Theorem \ref{MT2}. Although $G$ is assumed to be finitely generated in Theorem \ref{MT2}, we do not make that assumption for the domain of $f_i$ in \ref{standing assumption} since we need to consider the restriction of $f_i$ to subgroups of $G$ that are not known to be finitely generated. 

\begin{assumption}\label{standing assumption}
Let $G$ be a countable group and $\{f_i: G\rightarrow \pi_1(\mc{O}_{\bar n_i})\}$ be a sequence as in Theorem \ref{MT2}. To abbreviate notation, we denote the natural projection from $\pi_1(\mc{O})$ to $\pi_1(\mc{O}_{\bar n_i})$ by $\pi^i$. Fix discrete embeddings $e:\pi_1(\mc{O})\rightarrow\Gamma\subset{\rm PSL}(2, \mathbb{R})$ and $\{e_i:\pi_1(\mc{O}_{\bar n_i})\rightarrow\Gamma_i\subset{\rm PSL}(2, \mathbb{R})\}$ with the following properties:
\begin{enumerate}
\item $\{e_i\circ\pi^i\}$ converges algebraically to $e$ in ${\rm Hom}(\Gamma, {\rm PSL}(2, \mathbb{R}))$.

\item $\Gamma_i$ converges to $\Gamma$ in ${\rm PSL}(2, \mathbb{R})$ geometrically.

\item All punctures of $\mc{O}$ and $\mc{O}_{\bar n_i}$ correspond to parabolic subgroups of $\Gamma$ and $\Gamma_i$, respectively. 
\end{enumerate}
(See \cite[Theorem 4.3.2]{Katok} and \cite{J-M} for detail.) We identify $\pi_1(\mc{O})$ with $\Gamma$ and $\pi_1(\mc{O}_{\bar n_i})$ with $\Gamma_i$. By abuse of notation, we also say $\{\pi^i\}$ converges algebraically to $I_{\Gamma}$, the identity map on $\Gamma$. 
\end{assumption}

Before we explain the idea of the proof of Theorem \ref{MT2}, we make a couple reductions. The idea of limit group allows us to make the first reduction.


Recall that a non-principal ultrafilter $\omega$ is a finitely additive probability measure $\omega\colon 2^\N\to\{0, 1\}$ so that $\omega(F)=0$ for any finite set $F$. A statement $P(i)$ depending on an index $i$ is said to hold {\em \walmost surely} if $\omega(\{i\mid P(i) \text{ holds}\})=1$. Fix a non-principal ultrafilter $\omega$.  (See \cite{BH} for basic properties of ultrafilter.) Suppose $G$ is a countable group and $\{f_i:G\rightarrow {\rm PSL}(2, \mathbb{R})\}$ be a sequence of discrete representatives.  

\begin{definition}\label{omega kernel}
The $\omega$-kernel of $\{f_i\}$, denoted by $ker^\omega(f_i)$, is defined by 
\begin{equation*}
ker^\omega(f_i)=\{g\in G\mid f_i(g)=1 \hspace{2mm} {\text  \walmost surely}\}
\end{equation*}
$L=G/ker^\omega(f_i)$ is called the {\em limit group} associated to $\{f_i\}$. If $G$ is finitely generated, we call $L$ a {\em ${\rm PSL}(2, \mathbb{R})$-discrete limit group}. 
\end{definition}

\begin{lemma}\label{factorthroughlimitgroup}
Suppose $L=G/ker^\omega(f_i)$ is a ${\rm PSL}(2, \mathbb{R})$-discrete limit group associated to $\{f_i: G\rightarrow {\rm PSL}(2, \mathbb{R})\}$. Then \walmost surely, $f_i$ factors through the natural quotient map from $G$ to $L$. If $f'_i:L\rightarrow {\rm PSL}(2, \mathbb{R})$ is the factoring map, then $ker^\omega(f'_i)=\{1\}$. 



\end{lemma}

\begin{proof}
Note that ${\rm PSL}(2, \mathbb{R})$ is equationally Noetherian. Hence first statement follows from the proof of \cite[Corollary 6.3]{R}. We leave the second statement as an exercise for the reader. 
\end{proof}

Supppose $\{f'_i\}$ satisfies the conclusion of Theorem \ref{MT2}. Then one can precompose the factoring maps for $\{f'_i\}$ by the quotient map from $G$ to $L$ to get factoring maps for $\{f_i\}$.  So $\{f_i\}$ also satisfies the conclusion of Theorem \ref{MT2}. Hence by Lemma \ref{factorthroughlimitgroup} we have our first reduction:

\begin{lemma}\label{l:reduction 1}
If Theorem \ref{MT2} holds with the addition assumption that $ker^\omega(f_i)=\{1\}$, then Theorem \ref{MT2} holds. 
\end{lemma}

The following are some facts about ${\rm PSL}(2, \mathbb{R})$-discrete limit groups that we need later in the paper. 

\begin{lemma}\label{lgs}
Any sequence of proper epimorphisms between ${\rm PSL}(2, \mathbb{R})$-discrete limit groups has finite length. 
\end{lemma}

\begin{proof}
Since ${\rm PSL}(2, \mathbb{R})$ is linear, by \cite[Theorem B1]{BMR}, ${\rm PSL}(2, \mathbb{R})$ is equationally Noetherian. Hence by \cite[Corollary 6.2]{R} and \cite[Corollary 6.3]{R}, we know that a sequence of epimorphisms between ${\rm PSL}(2, \mathbb{R})$-limit groups eventually stabilizes. 
\end{proof}

Recall that a group $G$ is {\em commutative transitive} if $[g,h]=[g', h]=1$ implies $[g,g']=1$ for any $g, g', 1\neq h\in G$. 

\begin{lemma}\label{l:comm tran}
${\rm PSL}(2, \mathbb{R})$-discrete limit groups are commutative transitive. 
\end{lemma}

\begin{proof}
Let $\{f_i: G\rightarrow {\rm PSL}(2, \mathbb{R})\}$ be a sequence discrete of representations. Let $L=G/ker^\omega(f_i)$ be the associated limit group. Suppose $g, g', 1\neq h\in L$ satisfy $[g,h]=[g', h]=1$. Let $\widetilde{g}, \widetilde{g'}$ and $\widetilde{h}$ be lifts in $G$ of $g, g'$ and $h$, respectively. Note that $f_i(\widetilde{h})\neq 1$ \walmost surely since $h\neq 1$.  Observe that $[\widetilde{g}, \widetilde{h}], [\widetilde{g'}, \widetilde{h}]$ and $[\widetilde{g},\widetilde{g'}]$ are lifts of $[g,h], [g', h]$ and $[g,g']$, respectively. Since $[g, h]=1$, we have $f_i([\widetilde{g}, \widetilde{h}])=[f_i(\widetilde{g}), f_i(\widetilde{h})]=1$ \walmost surely. Similarly, we have $[f_i(\widetilde{g'}), f_i(\widetilde{h})]=1$ \walmost surely. Note that the image of $f_i$ in ${\rm PSL}(2, \mathbb{R})$ is commutative translation. As a result,  $[f_i(\widetilde{g}), f_i(\widetilde{g'})]=1$ \walmost surely, which implies that $[g, g']=1$.
\end{proof}

Here is our second reduction. 

\begin{lemma}\label{l:reduction 2}
Theorem \ref{MT2} holds if $G$ is abelian.
\end{lemma}

\begin{proof}
By Lemma \ref{l:reduction 1}, we can assume $ker^\omega(f_i)=\{1\}$. Hence $G$ is ${\rm PSL}(2, \mathbb{R})$-discrete limit group.  Since $G$ is finitely generated and abelian, $f_i(G)$ is abelian and $G=\mathbb{Z}^n\oplus F$, where $F$ is a finite abelian group. Hence $f_i(G)$ is cyclic. If $f_i(G)$ is infinite or has order bounded independent of $i$, then there is a subgroup of $\Gamma$ mapped isomorphically to $f_i(G)$ by $\pi^i$. Hence in this case, $f_i$ factors through $\Gamma$. So we are left with the case where $f_i(G)$ has order going to infinity. In this case, up to conjugation, $f_i(G)$ is contained in an elliptic subgroup $C_i$ of $\Gamma_i$ and $C_i$ converges to a parabolic subgroup $P$ of $\Gamma$ corresponding to a puncture. Note that $P$ is abelian and is mapped onto $C_i$ by $\pi^i$. Therefore $f_i|_{\mathbb{Z}^n}$ factor through $P$ and hence $\Gamma$ since $\mathbb{Z}^n$ is free abelian. Since $ker^\omega(f_i)=\{1\}$, $f_i$ is injective on any finite subset of $G$ \walmost surely. Hence $f_i|_F$ is injective, which implies that $F=\mathbb{Z}/a\mathbb{Z}$ for some $a$. Therefore $f_i|_F$ and hence $f_i$ factors through $\Gamma*_P(P\oplus \mathbb{Z}/a\mathbb{Z})$. 
\end{proof}

By Lemma \ref{l:reduction 1} and Lemma \ref{l:reduction 2}, we make the following additional assumption in Section \ref{s:Bounded} to Section \ref{proof}.  

\begin{assumption}\label{SA2}
Given \ref{standing assumption}, we assume $G$ is not abelian and $ker^\omega(f_i)=\{1\}$. 
\end{assumption}

We introduce one more important ingredient before we explain the idea of the proof of Theorem \ref{MT2}. 

\begin{definition}\label{d:translation length}
Let $S\subset G$ be a finite subset of $G$. For any homomorphism $f: G\rightarrow {\rm PSL}(2, \mathbb{R})$,  define the {\em translation length} $|f|_S$ of $f$ with respect to $S$ by 
\begin{equation*}
|f|_S=inf_{x\in\mathbb{H}}\{max_{g\in S}\{d_{\mathbb{H}}(f(g)x, x)\}\}.
\end{equation*}
A point $x\in \mathbb{H}$ realizing $|f|_S$ is called a {\em centrally located} point of $f$ with respect to $S$.
\end{definition}

\begin{lemma}\label{l:cent located points exits}
Under Assumption \ref{standing assumption} and Assumption \ref{SA2}, a centrally located point of $f_i$ with respect to any finite subset $S$ of $G$ exists \walmost surely. 
\end{lemma}

\begin{proof}
By  \cite[Proposition 2.1]{Bestvina}, it suffices to show that $f_i(G)$ is not contained in a parabolic subgroup \walmost surely. Suppose $f_i(G)$ is contained in a parabolic subgroup \walmost surely. Let $g,h\in G$. Then $f_i([g,h])=[f_i(g), f_i(h)]=1$ \walmost surely since parabolic subgroups are abelian. Hence $[g, h]\in ker^\omega(f_i)$. Therefore $[g, h]=1$. As a result, $G$ is abelian. Contradiction. 
\end{proof}

We now explain the idea of the proof of Theorem \ref{MT2}: 
Assume \ref{standing assumption} and \ref{SA2}. Suppose $S$ generates $G$. The proof of Theorem \ref{MT2} is separated into the following two cases: 
\begin{enumerate}
\item $\{|f_i|_S\}$ has a bounded subsequence.
\item $\{|f_i|_S\}$ has no bounded subsequence.
\end{enumerate}

Case (1) is the easy case. In this case, after passing to subsequences and pos-composing with conjugation in $\Gamma_i$,  $\{f_i\}$ converges to a representation $f: G\rightarrow \Gamma$. We then observed that $f_i$ factors through $f$ for all big $i$. (This is explained in Section \ref{s:Bounded})

Case (2) is the hard case. The goal is to reduce it to Case (1). Note that if $\alpha_i$ is an automorphism of $G$ and $f_i\circ\alpha_i$ factors through $\Gamma$ (or $\Gamma*_{\bar P}\bar A$), then so does $f_i$.  This suggests the following approach: we precompose $f_i$ with automorphism $\alpha_i$ of $G$ to ``shorten'' the translation length $|f_i|_S$, i.e. so that $|f_i\circ \alpha_i|_S<|f_i|_S$.  If the resulting sequence has a bounded subsequence, we are done by Case (1). However, in general, the automorphisms of $G$ do not have enough ``shortening power" to turn $\{|f_i\circ\alpha_i|_S\}$ into a bounded sequence. This is where Sela's theory of limit groups and shortening argument comes in: If $f_i\circ\alpha_i$ is ``shortest'' possible (Definition \ref{d:short}) and $\{|f_i\circ\alpha_i|_S\}$ is still unbounded, then Sela's shortening argument (IF it applies) says that $L_1=G/ker^\omega(f_i\circ\alpha_i)$ is a proper quotient of $G$ and $f_i\circ\alpha_i$ factors through the quotient map $G\rightarrow L_1$ \walmost surely. Let $f_i^1: L_1\rightarrow \Gamma_i$ be the factoring map. Then it suffice to show that $\{|f_i^1|_{S_1}\}$ has a bounded subsequence, where $S_1$ is a finite generating set of $L_1$. Now we can repeat the above process and the automorphisms of $L_1$ give us extra ``shortening power", which gives us a better chance of attaining a subsequence of bounded translation lengths. The above process gives us sequences $\{f_i^1\}, \{f_i^2\}, \dots$, which are more and more likely to be ``bounded'' , i.e. have subsequence with bounded translation lengths. At the same time, a sequence of proper quotients $G\rightarrow L_1\rightarrow \cdots $ is produced. By Lemma \ref{lgs} such a sequence has finite length. As a result, $\{f_i^J\}$ has a subsequence with bounded translation lengths for some $J$ and Case (2) is reduced to Case (1).

However, Sela's shortening argument does not work in our setting, at least not directly. 
Roughly, here is how the shortening argument works: Suppose $L_1=G/ker^\omega(f_i\circ\alpha_i)$ is NOT a proper quotient of $G$. The sequence $\{f_i\}$ induces an action of $G$ on an ultralimit $\mc{X}$ of $(\mathbb{H}^2, \frac{1}{|f_i|_S}d_{\mathbb{H}})$. If $\{|f_i|_S\}$ is unbounded, $\mc{X}$ is an $\mathbb{R}$-tree. Let $\mc{T}$ be the minimal $G$-invariant subtree of $\mc{X}$. (The construction of $\mc{X}$ and $\mc{T}$ and their properties are explained in Section \ref{rtree}. ) If the action of $G$ on $\mc{T}$ is ``nice'', then the Rips machine (\cite[Theorem 5.1]{Guirardel}) can be applied to obtain useful information about $G$ and its action on $\mc{T}$, which allows us to find automorphisms of $G$ that shorten $|f_i\circ \alpha_i|_S$. But this is impossible as $f_i\circ \alpha_i$ is the ``shortest'' by construction. Hence $L_1$ is a proper quotient of $G$. 

The main difficulty that we need to overcome is that the action of $G$ on $\mc{T}$ is not ``nice'': there are subtrees of $\mc{T}$ that can not be described by the output of the Rips machine. These ``bad'' subtrees (with the actions of their stabilizers in $G$) do not seem to have the dynamical properties that allow the shortening argument to work. (Properties of these ``bad'' subtrees will be studied in Section \ref{es}.)  To deal with these ``bad'' subtree, we apply a construction introduced by Groves, Manning and Wilton to ``separates'' the ``bad'' subtrees from the rest of the tree $\mc{T}$. More specifically, we obtain a splitting $\mathbb{G}$ of $G$ with two types of vertex groups, one of which corresponds to the stabilizers of the $G$-orbits of  ``bad'' subtrees. We call this type of vertex groups ``bad''. (The Groves-Manning-Wilton construction will be explained in Section \ref{es}.) While these ``bad'' vertex groups are stabilizers of ``bad'' subtrees of $\mc{T}$, they are easy to understand algebraically: they are abelian. So we can restrict $f_i$ to these ``bad'' vertex groups and see that they factor through groups of the form $\mathbb{Z}\oplus(\mathbb{Z}/a\mathbb{Z})$. Now we restrict $f_i$ to the vertex groups of $\mathbb{G}$ that are not ``bad'' and try to run the shortening argument on each of these vertex groups. There are two problems that one needs to deal with: 
\begin{enumerate}
\item The ``not-bad'' vertex groups are not known to be finitely generated and translation length is only well defined for finite subsets. 
\item New ``bad'' subtrees might show up again. 
\end{enumerate} 
First we need to show that it makes sense to run the shortening argument in these vertex groups of $\mathbb{G}$, even though they are not known to be finitely generated. We do this in Section \ref{s:rf}. To deal with the second problem, when ``bad'' subtrees show up again, we apply the Groves-Manning-Wilton construction again to get a refinement of $\mathbb{G}$ and then work with the new (smaller) ``not-bad'' vertex groups. The key here is to show that new ``bad'' subtrees do not keep showing up forever. This is explained in Section \ref{msas}. Then we will be able to obtained a refinement $\mathbb{G}$ such that when one tries to run the shortening argument on the ``not-bad'' vertex groups of $\mathbb{G}$ , ``bad'' subtrees will not show up. We then show that the shortening argument works in this case (see Section \ref{hs}). In Section \ref{proof}, we will put all the above ideas together to finish the proof of Theorem \ref{MT2}.

\section{\large Bounded translation length}\label{s:Bounded}

\begin{theorem}\label{strongversionofboundedcase}
Assume \ref{standing assumption} and \ref{SA2}. Suppose $G$ is finitely generated and let $S$ be a finite generating set of $G$. Suppose ${\rm lim}^\omega |f_i|_S< \infty$ . Then there exists $f': G\rightarrow \Gamma$ such that $f_i=\pi^i\circ \tilde c_i\circ f'$ \walmost surely, where $\tilde c_i$ is an inner automorphism of $\Gamma$.

\end{theorem}

Here are some basic facts that we need. 

\begin{lemma}\label{econp}
Let $H_1$ and $H_2$ be two parabolic subgroups of $\Gamma$. Suppose $\pi^i(H_1)$ and $\pi^i(H_2)$ are conjugate to each other by $\gamma\in \Gamma_i$. Then $H_1$ and $H_2$ are conjugate to each other by some $\tilde \gamma\in \pi^{-1}_i(\gamma)$  
\end{lemma}

The next lemma follows from the two properties of $e_i$ and $e$ in Assumption \ref{standing assumption} and \cite[Proposition 3.8]{J-M}

\begin{lemma}\label{converging representations}
For any $\gamma\in \Gamma$, we have $\pi^i(\gamma)$ converges in ${\rm PSL}(2, \mathbb{R})$ to $\gamma$.
\end{lemma}

The next lemma follows easily from \cite[Proposition 3.10]{J-M}.

\begin{lemma}\label{isolatednbh}
For any $\gamma\in \Gamma$, there exists a neighborhood $U_{\gamma}$ of $\gamma$ in ${\rm PSL}(2, \mathbb{R})$ such that if $\pi^i(\gamma')\in U_{\gamma}$, then $\gamma'=\gamma$. 
\end{lemma}

Denote the Margulis constant for $\mathbb{H}$ by $\epsilon_2$. Let $H$ be a maximal elliptic or parabolic subgroup of a finitely generated Fuchsian group $\Gamma$. We define 
\begin{equation*}
D(H)=\{x\in \mathbb{H}\mid d(x, hx)\leq \epsilon_2 \text{ for some } 1\neq h\in H\}
\end{equation*}

We call $D(H)$ the {\em Margulis domain }associated to $H$. The next lemma are standard facts about Margulis domain. 

\begin{lemma}\label{l:margulis}
Let $D(H)$ be the Margulis domain for a maximal elliptic or parabolic subgroup $H$ of a finitely generated Fuchsian group $\Gamma$. 
\begin{enumerate}
\item $D(H)$ is a ball (or horoball) centered at the fixed point of $H$. 
\item\label{mag2} For any $h\in H$ and any $x\in \partial D(H)$, we have $d(x, hx)\geq\epsilon_2$. 
\item\label{mag3} For any $\gamma\in\Gamma$ and any $x\in D(H)$, if $\gamma x\in D(H)$, then $\gamma\in H$. 
\end{enumerate}
\end{lemma}

\begin{lemma}\label{boudnary disct}
Assume \ref{standing assumption}. Suppose $G$ is finitely generated and let $S$ be a finite generating set of $G$. Suppose ${\rm lim}^\omega |f_i|_S< \infty$.  Let $x_i$ be a centrally located point for $f_i$ with respect to $S$.Let $H_i$ be a maximal elliptic (or parabolic) subgroup of $\Gamma_i$ and $D_i=D(H_i)$ be the associated Margulis domain. Suppose at least one of the following is true. 
\begin{enumerate}
\item There exists $g\in G$ such that $f_i(g)\in H_i$ \walmost surely and $x_i\notin D_i$.
\item There exists $g\in G$ such that $f_i(g)\notin H_i$ \walmost surely and $x_i\in D_i$.
\end{enumerate}
Then $d_{\mathbb{H}}(x_i, \partial D_i)$ has finite $\omega$-limit.
\end{lemma}

\begin{proof}
Since ${\rm lim}^\omega |f_i|_S< \infty$ and $S$ generate $G$, we know that $d_{\mathbb{H}}(x_i,  f_i(g)x_i)$ has finite $\omega$-limit for any $g\in G$. 

Suppose (1) is true.  Lemma \ref{l:margulis} (\ref{mag2}) implies that \walmost surely $d_{\mathbb{H}}(x_i, f_i(g)x_i)$ is greater than $2d_{\mathbb{H}}(x_i, \partial D_i)$ with an error bounded independent of $i$. Therefore we know that $d_{\mathbb{H}}(x_i, \partial D_i)$ has finite $\omega$-limit. 

Now suppose (2) is true. By Lemma \ref{l:margulis}(\ref{mag3}) $f_i(g)x_i$ is outside of $D_i$. Therefore $d_{\mathbb{H}}(x_i, f_i(g)x_i)\geq d_{\mathbb{H}}(x_i, \partial D_i)$. Since $d_{\mathbb{H}}(x_i, f_i(g)x_i)$ has finite $\omega$-limit, so does $d_{\mathbb{H}}(x_i, \partial D_i)$.   
\end{proof}

\begin{proof}[Proof of Theorem \ref{strongversionofboundedcase}]
Let $x_i$ be a centrally located point for $f_i$ with respect to $S$. Choose fundamental polyhedras in the hyperbolic plane $\mathbb{H}$ for $\{\Gamma_i\}$ and $\Gamma$, denoted by $\{\Omega_i\}$  and $\Omega$, respectively with the following properties:
\begin{enumerate}
\item $\Omega_i\subset \Omega_{i+1}$.

\item $\Omega_i\subset \Omega$. 

\item $\Omega_i$ converges to $\Omega$ uniformly on compact subsets of $\mathbb{H}$. 
\end{enumerate}
(See \cite[Theorem 4.3.2]{Katok} for more detail of the construction of $\Omega_i$  and $\Omega$. )
For each $i$, pick $\gamma_i\in \Gamma_i$ such that $\gamma_i x_i\in \Omega_i$. Let $c_{i}: \Gamma_i\rightarrow \Gamma_i$ be defined by $c_i(\beta)=\gamma_i\beta\gamma_i^{-1}$. Then $x_i'=\gamma_ix_i$ is a centrally located point for $f_i'=c_i\circ f_i$ and $|f_i'|_S=|f_i|_S$, which implies $\{|f_i'|_S\}$ is bounded. 

For each $i$, let $H_i$ be a maximal elliptic (or parabolic) subgroup of $\Gamma_i$ and $D_i=D(H_i)$ be the associated Margulis domain. We claim that there exists $g\in G$ such that $f_i(g)\notin H_i$ \walmost surely. (Otherwise for any $g\in G$ we have $f_i(g)\in H_i$ \walmost surely. Then $G=G/ker^\omega(f_i)$ is abelian since $H_i$ is abelian. Contradiction.) Then it follows from Lemma \ref{boudnary disct} that either $x_i$ is outside of $D_i$ or $d_{\mathbb{H}}(x_i, \partial D_i)$. This implies that $\{x'_i\in\mathbb{H}\mid i=1,2,\dots\}$ has a bounded subsequence in $\mathbb{H}$. 
Hence for each $g\in S$ the sequence $\{f'_i(g)\}$ lies in a compact set. Therefore pass to subsequence if necessary,  ${\rm lim}f'_i(g)$ exists. Define a map $f':S\rightarrow {\rm PSL}(2,\mathbb{R})$ by $f'(g)={\rm lim}f'_i(g)$. It is easy to see that $f'$ extends to a homomorphisms from $G$ to ${\rm PSL}(2,\mathbb{R})$. Since $\Gamma$ is the geometric limit of $\Gamma_i$, we know that $f'(g)={\rm lim}f'_i(g)\in \Gamma$. Therefore $f'$ is a homomorphism from  $G$ to $\Gamma$. 

We now show that $f'_i=\pi^i\circ f'$ for infinitely many $i$. For $g\in S$, we have $f'_i(g)\in U_{f'(g)}$ for all large $i$, where $U_{f'(g)}$ is a neighborhood of $f'(g)$ satisfying the conclusion of Lemma \ref{isolatednbh}. For each $i$, since $\pi^i$ is surjective, there exists $\beta_i\in \Gamma$ with $\pi^i(\beta_i)=f'_i(g)$. Hence we have $\pi^i(\beta_i)\in U_{f'(g)}$. By Lemma \ref{isolatednbh}, we have $\beta_i=f'(g)$ for all large $i$.  Hence  $\pi^i(f'(g))=f'_i(g)$. Since $S$ generates $G$, we have $(\pi^i\circ f')(g)=f'_i(g)$ for all $g\in G$ and all large $i$.  

For each $i$, let $\tilde c_i:\Gamma\rightarrow \Gamma$ be a lift a $c_i$. Then for large $i$ we have 
\begin{equation}\label{commute}
f_i=\pi^i\circ \tilde c_i\circ  f'.
\end{equation} 
\end{proof}

\vspace{3mm}

\section{\large The $\mathbb{R}$-tree $\mathcal{T}$}\label{rtree}

We assume \ref{standing assumption} and \ref{SA2} in this section. 

We say $1\neq g\in G$ is {\em elliptic, parabolic or hyperbolic} (with respect to $\{f_i\}$) if \walmost surely $f_i(g)$ is elliptic, parabolic or hyperbolic, respectively. We say that an abelian subgroup $H$ of $G$ is a {\em elliptic, parabolic or hyperbolic} (with respect to $\{f_i\}$) if for some $g\in H$ (and hence all $g\in H$) \walmost surely $f_i(g)$ is elliptic, parabolic or hyperbolic, respective.  

In this section, we assume \ref{standing assumption} and that $G$ is finitely generated relative to a finite collection of pairwise non-conjugate elliptic or parabolic subgroups $\{H_1, \dots, H_n\}$.  Let $S$ be a finite subset of $G$ such that $S\cup H_1\cup \cdots\cup H_n$ generates $G$. In this section we assume ${\rm lim}^\omega|f_i|_S=\infty$. 

All the asymptotic cones in this paper are defined using $\omega$. Let $x_i$ be a centrally located point of $f_i$ with respect to $S$. Note that the existence of $x_i$ follows from Lemma \ref{l:cent located points exits}.  Let $\mathcal{X}$ be an asymptotic cone of the hyperbolic plane $\mathbb{H}$ defined by $\{(\mu_i, x_i)\}$. Let $x\in \mathcal{X}$ be the point defined by $\{x_i\}$. Note that $\mathcal{X}$ is an $\mathbb{R}$-tree. Denote by $d$ the metric of $\mathcal{X}$, $d_{\mathbb{H}}$ the metric of $\mathbb{H}$ and let $d_i=\frac{1}{\mu_i}d_{\mathbb{H}}$. 

Since $S$ is not a generating set for $G$, the actions of $G$ on $\mathbb{H}$ do not in general induce an action of $G$ on $\mathcal{X}$: When $g\in G$ is not a product of elements in $S$, it might take a sequence $\{y_i\}$ with $d_i(x_i, y_i)$ bounded to a sequence $\{f_i(g)y_i\}$ such that $d_i(x_i, f_i(g)y_i)$ is not bounded. In this section, we consider only the case where $g\cdot \{y_i\}=\{f_i(g)y_i\}$ does induce a well defined action of $G$ on $\mc{X}$, i.e. for any $g\in G$ and any sequence of points $\{y_i\}$ in $\mathbb{H}$, we have that  ${\rm lim}^{\omega}d_i(x_i, f_i(g)y_i)$ is finite as long as ${\rm lim}^{\omega}d_i(x_i, y_i)$ is finite. Denote by $ker(\mathcal{X})$ the kernel of the action of $G$ on $\mathcal{X}$.





\vspace{3mm}

\begin{lemma}\label{arc sta abel}
The arc stabilizers of $\mathcal{X}$ in $L$ are abelian. 
\end{lemma}

\begin{proof}
Let  $l>0$ be a lower bound on the translation lengths of hyperbolic elements of $\Gamma$. Note that when $i$ is big enough, $\frac{1}{2}l$ is a lower bound on the translation lengths of hyperbolic elements of $\Gamma_i$. Let $l'>0$ be a number such that $d_{\mathbb{H}}(x, hx)>l'$ for any maximal elliptic (parabolic) subgroup $\tilde H_i$ of $\Gamma_i$, any $h\in \tilde H_i$ and any $x$ outside of the Margulis domain $D(\tilde H_i)$. Let $e=min\{\frac{1}{1000}l, \frac{1}{1000}l'\}$.

Let $J\subset \mathcal{X}$ be a non-trivial segment and let $g$ be an element in the stabilizer of $J$. Let $y=[\{y_i\}], z=[\{z_i\}]\in \mathcal{X}$ be the end points of $J$. For each $i$, consider the geodesic quadrilateral $[y_i, z_i, f_i(g)y_i, f_i(g)z_i]$ in $\mathbb{H}$. Let $L=d(y, z)$. Let $L_i=d_i(y_i, z_i)$, $\epsilon_i=d_i(y_i, f_i(g)y_i)$ and $\epsilon'_i=d_i(z_i, f_i(g)z_i)$. Then we have:
\begin{equation}\label{Kapovich1}
d_{\mathbb{H}}(f_i(g)y_i, f_i(g)z_i)=d_{\mathbb{H}}(y_i, z_i)=\mu_iL_i;
\end{equation}
\begin{equation}\label{Kapovich2}
d_{\mathbb{H}}(y_i, f_i(g)y_i)=\mu_i\epsilon_i;
\end{equation}
\begin{equation}\label{Kapovich3}
d_{\mathbb{H}}(z_i, f_i(g)z_i)=\mu_i\epsilon'_i. 
\end{equation}
By the definition of $\mathcal{X}$ we have:
\begin{equation}\label{Kapovich4}
{\rm lim}^{\omega}L_i=L>0; 
\end{equation}
\begin{equation}\label{Kapovich5}
{\rm lim}^{\omega}{\epsilon_i}={\rm lim}^{\omega}{\epsilon'_i}=0; 
\end{equation}
By (\ref{Kapovich1}),(\ref{Kapovich2}),(\ref{Kapovich3}),(\ref{Kapovich4}),(\ref{Kapovich5}) and the fact that $\mu_i\rightarrow \infty$ as $i\rightarrow \infty$, \walmost surely \cite[Lemma 3.10]{Kapovich} can be applied to the quadrilateral $[y_i,z_i, f_i(g)y_i, f_i(g)z_i]$. Hence \walmost surely, there are segments $J_i\subset [y_i, z_i]$ and $J_i'\subset[f_i(g)y_i, f_i(g)z_i]$ with length $\frac{1}{2}\mu_iL$ such that the Hausdorff distance between $J_i$ and $J_i'$ is less than $e$. Hence up to an error of $4e$, $g$ acts on $J_i$ as translations. Therefore for any $g_1, g_2$ in the stabilizer of $J$, both of the commutators $[g_1,g_2]$ and $[g_2, g_1^{-1}]=[g_1,g_2]^{g_{1}}$ move a point $u_i$ in the middle of $J_i$ by a distance less than $16e$. By the definition of $l$ and $e$, we know that  $f_i([g_1,g_2])$ moves $u_i$ by a distance less than $l$, which implies that $f_i([g_1, g_2])$ is not hyperbolic. 
  
Suppose $f_i([g_1, g_2])$ is elliptic (parabolic) and lies in a maximal elliptic (parabolic) subgroup $\tilde H_i$. Let $w_i$ be the fixed point of  $\tilde H_i$. Then $w'_i=f_i(g^{-1}_1)w_i$ is the fixed point of $(\tilde H_i)^{g_1}$. Let $D(\tilde H_i)$ and $D((\tilde H_i)^{g_1})$ be Margulis domain associated to $\tilde H_i$ and $(\tilde H_i)^{g_1}$, respectively. Since both $[g_1,g_2]$ and $[g_2, g_1^{-1}]=[g_1,g_2]^{g_{1}}$ move the point $u_i$ by less than $16e$, which is smaller than $l'$, we have $u_i\in D(\tilde H_i)\cap D((\tilde H_i)^{g_1})$. On the other hand, by Lemma \ref{l:margulis}, $D(\tilde H_i)$ and $D((\tilde H_i)^{g_1})$ are disjoint if $w_i\neq w'_i$. Therefore we have $w_i=w'_i=f_i(g^{-1}_1)w_i$. Hence $f_i(g_1)$ is elliptic (parabolic) with fixed point $w_i$. Therefore $f_i(g_2g^{-1}_1g^{-1}_2)$ is also elliptic (parabolic) with fixed point $w_i$ and so is $f_i(g_2)$. So $f_i([g_1, g_2])=1$ \walmost surely, which implies that $[g_1, g_2]\in ker^\omega(f_i)=\{1\}$. Hence the proof of the lemma is complete. 
\end{proof}

Let $\mathcal{T}$ be the minimal $L$-invariant subtree of $\mathcal{X}$ containing the base point $x$ of $\mathcal{X}$.

\begin{lemma}\label{tripod sta elliptic}
If $H_{T}\subset L$ is the stabilizer of a non-degenerate tripod $T$ in $\mathcal{T}$. Then $H_{T}$ is either elliptic or parabolic. 
\end{lemma}

\begin{proof}
Let $h\in H_T$. Since $H_T$ is abelian by Lemma \ref{arc sta abel}, it suffices to show that $h$ is elliptic or parabolic. Let $a=[\{a_i\}]$, $b=[\{b_i\}]$ and $c=[\{c_i\}]$ be the endpoints of $T$ and $u$ be the three valence vertex in $T$. Let $L=d(a, u)$, $M=d(b,u)$ and $N=d(c,u)$. Since $T$ is non-degenerate,  $L$, $M$ and $N$ are all positive. 

Let $u_i$ be the point on the geodesic $[a_i, b_i]$ minimizing the distance to $c_i$, or equivalently, the projection of $c_i$ onto $[a_i, b_i]$. Then $f_i(h)u_i$ is the projection of $f_i(h)c_i$ onto $[f_i(h)a_i, f_i(h)b_i]$.  Since $\mathbb{H}$ is $\delta$-hyperbolic and the hyperbolic constant $\delta_i$ of $(\mathbb{H}, d_i)$ goes to zero as $i$ goes infinity, we have $u=[\{u_i\}]$, which implies that the \w-limits of $L_i=d_i(a_i, u_i)$, $M_i=d_i(b_i, u_i)$ and $N_i=d_i(c_i, u_i)$ are $L$, $M$ and $N$, respectively.  Also, since $h$ fixes the tripod $T$, the \w-limits of $\epsilon_i=d_i(a_i, f_i(h)a_i)$, $\epsilon'_i=d_i(b_i, f_i(h)b_i)$ and $\epsilon''_i=d_i(c_i, f_i(h)c_i)$ are zero. By definition of $d_i$, we have $d_{\mathbb{H}}(a_i, u_i)=L_i\mu_i$, $d_{\mathbb{H}}(b_i, u_i)=M_i\mu_i$, $d_{\mathbb{H}}(c_i, u_i)=N_i\mu_i$, $d_i(a_i, f_i(h)a_i)=\epsilon_i\mu_i$, $d_i(b_i, f_i(h)b_i)=\epsilon'_i\mu_i$ and $d_i(c_i, f_i(h)c_i)=\epsilon''_i\mu_i$. Hence \walmost surely the claim in the proof of \cite[Lemma 4.1]{RS1} can be applied to the two sets of points $\{a_i, b_i, c_i, u_i\}$ and $\{f_i(h)a_i, f_i(h)b_i, f_i(h)c_i, f_i(h)u_i\}$ to conclude that $d_{\mathbb{H}}(u_i, f_i(h)u_i)\leq 10\delta$ \walmost surely. Note that the above facts imply that \walmost surely \cite[Lemma 3.10]{Kapovich} can be applied to the quadrilaterals $[a_i, b_i,f_i(h)a_i, f_i(h)b_i]$ and $[c_i, u_i,f_i(h)c_i, f_i(h)u_i]$.

Applying \cite[Lemma 3.10]{Kapovich} to $[a_i, b_i,f_i(h)a_i, f_i(h)b_i]$,  we see that there are intervals $J_i\subset [a_i, b_i]$ and $J'_i\subset [f_i(h)a_i, f_i(h)b_i]$ such that (1) the Hausdorff distance in $\mathbb{H} $ between $J_i$ and $J'_i$ is arbitrarily small and (2) the ratio of the length of $J_i$ and $d_\mathbb{H}(a_i, b_i)$ is arbitrarily closed to one.  Note that (2) together with the fact that neither of the \w limts of $L_i$ and $M_i$ are zero implies that $u_i\in J_i$. Similarly, we have $f_i(h)u_i\in J'_i$.

Applying \cite[Lemma 3.10]{Kapovich} to $[c_i, u_i,f_i(h)c_i, f_i(h)u_i]$,  we see that there exists $v_i\in[c_i, u_i]$ and $v'_i\in [f_i(h)c_i, f_i(h)u_i]$ such that $d_\mathbb{H}(v_i, v'_i)$ is arbitrarily small. Note that  $u_i$ is also the projection of $v_i$ onto $J_i$. Similarly, $f_i(h)u_i$ is also the projection of $v'_i$ onto $J'_i$. Since the Hausdorff distance between $J_i$ and $J'_i$ is arbitrarily small and $d_\mathbb{H}(v_i, v'_i)$ is also arbitrarily small, we know that $d_{\mathbb{H}}(u_i, f_i(h)u_i)$ is arbitrarily small \walmost surely. Therefore \walmost surely $d_{\mathbb{H}}(u_i, f_i(h)u_i)$ is smaller than the translation length of all the hyperbolic elements in $\Gamma_i$. Hence $f_i(h)$ is elliptic or parabolic. 
\end{proof}

In general, non-degenerate tripods in $\mc{T}$ could have infinite stabilizers, which are either elliptic or parabolic subgroups of $G=L$ by Lemma \ref{tripod sta elliptic}.  In fact, there could be unstable arcs in $\mc{T}$ whose stabilizers are infinite elliptic or parabolic subgroups. The existence of such arcs prevents us from applying the Rips machine (\cite[Theorem 5.1]{Guirardel}) to get a ``nice'' graph of actions decomposition of the action of $G$ on $\mc{T}$, which is essential to the shortening argument.  Hence it is important to understand the subtree trees in $\mc{T}$ (and $\mc{X}$) fixed by elliptic or parabolic subgroup, which is the goal of the next section.

\vspace{3mm}

\section{\large elliptic(parabolic) splittings}\label{es}
In this section, with the same notations and assumptions as the previous section, we continue the study of the action of $G$ on $\mc{T}$ (and $\mc{X}$) and closely examine the subtree trees of $\mc{T}$ (and $\mc{X}$) fixed by elliptic or parabolic subgroup of $G$.  The main result is the existence of an abelian splitting of $G$ in the case when $\mathcal{T}$ has a non-degenerate segment with elliptic (parabolic) stabilizer (Theorem \ref{GMWtree}). We obtain this result by using a construction developed by Groves, Manning and Wilton in the setting of representations into ${\rm PSL}(2, \mathbb{C})$. 

Let $\delta>0$ be the smallest Gromov hyperbolic constant for $\mathbb{H}$. 

\begin{definition}\label{deltaball}
The {\em $\delta$-ball} ({\em $\delta$-horoball}) of an elliptic (parabolic) element $\gamma\in {\rm PSL}(2, \mathbb{R})$ is the closed ball (horoball) in $\mathbb{H}$ centered at the fixed point of $\gamma$ such that every point on the boundary of the ball is moved by $\delta$ by $\gamma$. 
For an elliptic (parabolic) element $h\in G$ , the {\em $\delta$-ball} of $h$ with respect to $f_i$, denoted by $B_i(h)$, is the $\delta$-ball ($\delta$-horoball) of $f_i(h)$. 
\end{definition}

Let $H$ be an elliptic (parabolic) subgroup and $1\neq h\in H$. Denote by $\xi_i$ the fixed point of $f_i(h)$ in $\mathbb{H}$ ($\partial \mathbb{H}$) whenever $f_i(h)$ is elliptic (parabolic).  
Let $\mathcal{X}$ be the asymptotic cone defined by $\{x_i\}$, $\{\mu_i=|f_i|_S\}$ and $\omega$ and suppose $\{f_i\}$ induces an action of $G$ on $\mathcal{X}$. Let ${\rm Fix}_{\mathcal{X}}(h)$ be the set of points in $\mathcal{X}$ fixed by $h$. 

\begin{lemma}\label{efl}
$y\in{\rm Fix}_{\mathcal{X}}(h)$ if and only if $y\in \mathcal{X}$ can be defined by $\{y_i\}$ where $y_i\in B_i(h)$ \walmost surely. 
\end{lemma}

\begin{proof}
Suppose $y$ is defined by $\{y_i\}$ where $y_i\in B_i(h)$ \walmost surely. Then by definition of $B_i(h)$, we have $d_{\mathbb{H}}(f_i(h)y_i, y_i)\leq \delta$ \walmost surely. Hence we know that $d(hy, y)={\rm lim}^{\omega}d_{\mathbb{H}}(f_i(h)y_i, y_i)/\mu_i={\rm lim}^{\omega}\delta/\mu_i=0$. So $h$ fixes $y$. 

Now suppose $h$ fixes $y$ and $y$ is defined by $\{y_i\}$ where $y_i\notin B_i(h)$ \walmost surely. For $y_i\notin B_i(h)$, let $\epsilon_i$ be the intersection of the segment $[y_i, \xi_i]$ with $\partial B_i(h)$. By Gromov hyperbolicity of $\mathbb{H}$, we have $d_{\mathbb{H}}(f_i(h)y_i, y_i)=2d_{\mathbb{H}}(y_i, \epsilon_i)+l_i$ for some $l_i$ satisfying $|l_i|\leq 4\delta$. Since $h$ fixes $y$, we know that ${\rm lim}^{\omega}d_{\mathbb{H}}(f_i(h)y_i, y_i)/\mu_i=0$. Therefore ${\rm lim}^{\omega}d_{\mathbb{H}}(y_i, \epsilon_i)/\mu_i=0$. As a result $y$ is also defined by $\{\epsilon_i\}$, where $\epsilon_i\in B_i(h)$ \walmost surely.  
\end{proof}

\begin{lemma}\label{pointsinas}
Let $y=[\{y_i\}]\in \mathcal{X}$. Let $\epsilon_i$ be the intersection of the infinite geodesic (ray) from  $\xi_i$ towards $y_i$ with $\partial B_i(h)$. 
Then $\{\epsilon_i\}$ defines a point in $\mathcal{X}$.   
\end{lemma}

\begin{proof}
Since $y=[\{y_i\}]\in \mathcal{X}$, it suffices to show ${\rm lim}^{\omega}d_{\mathbb{H}}(y_i, \epsilon_i)/\mu_i$ is finite. 

We first consider the case when $y_i$ is on the segment $[\xi_i, \epsilon_i]$ \walmost surely. Then $y_i\in B_i(h)$ \walmost surely. Suppose ${\rm lim}^{\omega}d_{\mathbb{H}}(y_i, \epsilon_i)/\mu_i=\infty$. Given an arbitrary point $z=[\{z_i\}]\in\mathcal{X}$. Let $d$ be the distance between $y$ and $z$ in $\mathcal{X}$. Then $d_{\mathbb{H}}(z_i, y_i)/\mu_i\leq d+1$ \walmost surely. Hence we have $z_i\in B_i(h)$ \walmost surely. Hence by Lemma \ref{efl}, $h$ fixes $z$. Since $z$ is an arbitrary point in $\mathcal{X}$, $h$ fixes $\mathcal{X}$ point-wise. This contradicts the fact that $G$ acts on $\mathcal{X}$ faithfully and $h\neq 1$. 

Now we suppose $\epsilon_i$ is on the segment $[\xi_i, y_i]$. By Gromov hyperbolicity of $\mathbb{H}$, we have $d_{\mathbb{H}}(f_i(h)y_i, y_i)=2d_{\mathbb{H}}(y_i, \epsilon_i)+l_i$, where $|l_i|\leq 4\delta$. Since $G$ acts on $\mathcal{X}$, we know that ${\rm lim}^{\omega}d_{\mathbb{H}}(f_i(h)y_i, y_i)/\mu_i$ is finite. So ${\rm lim}^{\omega}d_{\mathbb{H}}(y_i, \epsilon_i)/\mu_i$ is also finite.
\end{proof}

Suppose $h$ is elliptic. If ${\rm lim}^{\omega}d_i(\xi_i, x_i)$ is finite, then $\{\xi_i\}$ defines a point $\xi$ in $\mathcal{X}$. If ${\rm lim}^{\omega}d_i(\xi_i, x_i)=\infty$, the sequence of geodesic segments $\{[x_i, \xi_i]\}$ defines a geodesic ray in $\mathcal{X}$. We denote the point in $\partial \mathcal{X}$ determined by this geodesic ray by $\xi$.   
When $h$ is parabolic, the geodesic rays $[x_i, \xi_i]$ determine a geodesic ray in $\mathcal{X}$. Again we denote the point in $\partial \mathcal{X}$ determined by this geodesic ray by $\xi$. In all this case, we say $\xi$ is the point in $\mathcal{X}\cup \partial \mathcal{X}$ determined by $\{\xi_i\}$ and write $\xi=[\{\xi_i\}]$. 

Depending on whether $\xi\in\mathcal{X}$ or $\xi\in\partial\mathcal{X}$, ${\rm Fix}_{\mathcal{X}}(h)$ is either a ball (Lemma \ref{ball}) or a horoball (Lemma \ref{horoball}) in $\mathcal{X}$. 

\begin{lemma}\label{ball}
Let $\epsilon=[\{\epsilon_i\}]$, where $\epsilon_i\in \partial B_i(h)$. If $\xi=[\{\xi_i\}]$ is in $\mathcal{X}$, then ${\rm Fix}_{\mathcal{X}}(h)=D(\xi, R)$, the ball centered at $\xi$ with radius $R=d(\epsilon, \xi)$. 
\end{lemma}

\begin{proof}
Suppose $y\in {\rm Fix}_{\mathcal{X}}(h)$. By Lemma \ref{efl} $y=[\{y_i\}]$ with $y_i\in B_i(h)$ \walmost surely. Hence $d(y, \xi)={\rm lim}^{\omega}d_{\mathbb{H}}(y_i, \xi_i)/\mu_i\leq {\rm lim}^{\omega}d_{\mathbb{H}}(\epsilon_i, \xi_i)/\mu_i=R$. So $y\in D(\xi, R)$.

Suppose $y=[\{y_i\}]\in D(\xi, R)$. Then ${\rm lim}^{\omega}d_{\mathbb{H}}(y_i, \xi_i)/\mu_i=d(y, \xi)\leq R={\rm lim}^{\omega}d_{\mathbb{H}}(\epsilon_i, \xi_i)/\mu_i$. Hence there exists $\{l_j\}$ with ${\rm lim}_{j\rightarrow \infty}l_j=0$, such that for each $j$ we have $d_{\mathbb{H}}(y_i, \xi_i)/\mu_i\leq d_{\mathbb{H}}(\epsilon_i, \xi_i)/\mu_i+l_j$ \walmost surely. Hence for each $j$ we have we have $d_{\mathbb{H}}(y_i, \xi_i)\leq d_{\mathbb{H}}(\epsilon_i, \xi_i)+l_j\mu_i$ \walmost surely. Let $y^j_i$ be the point in geodesic $[y_i, \xi_i]$ which that is $l_j\mu_i$ away from $y_i$. Then we have $d_{\mathbb{H}}(y^j_i, \xi_i)=d_{\mathbb{H}}(y_i, \xi_i)-l_j\mu_i\leq d_{\mathbb{H}}(\epsilon_i, \xi_i)$. Hence $h$ fixes $y^j=[\{y^j_i\}]$ by Lemma \ref{efl}. On the other hand, we have $d(y, y^j)={\rm lim}^{\omega}d_{\mathbb{H}}(y_i, y^j_i)/\mu_i={\rm lim}^{\omega}l_j=l_j$. Since $l_j\rightarrow 0$ as $j\rightarrow \infty$, we know that $\{y^j\}$ is a sequence of points in ${\rm Fix}_{\mathcal{X}}(h)$ converging to $y$. Hence $y\in {\rm Fix}_{\mathcal{X}}(h)$ since ${\rm Fix}_{\mathcal{X}}(h)$ is closed. 
\end{proof}

Suppose $\xi=[\{\xi_i\}]\in \partial \mathcal{X}$. Let $B: \mathcal{X}\rightarrow \mathbb{R}$ be the Busemann function defined by the geodesic ray from the base point $x=[\{x_i\}]$ to $\xi$ with $B(\xi)=-\infty$.

\begin{lemma}\label{horoball}
Let $\epsilon=[\{\epsilon_i\}]$, where $\epsilon_i\in \partial B_i(h)$. If $\xi$ is in $\partial\mathcal{X}$, then ${\rm Fix}_{\mathcal{X}}(h)=B^{-1}((-\infty, s])$, where $s=B(\epsilon)$. 
\end{lemma}

\begin{proof}
Let $[y, \xi]$ be the geodesic ray from $y$ to $\xi$ and $[\epsilon, \xi]$ be the geodesic ray from $\epsilon$ to $\xi$.  

Suppose $y\in {\rm Fix}_{\mathcal{X}}(h)$.  By Lemma \ref{efl} $y=[\{y_i\}]$ with $y_i\in B_i(h)$ \walmost surely. Let $v=[\{v_i\}]$ be a point in$[\epsilon, \xi]\cap[y, \xi]$ where $v_i$ is chosen to be on $[y_i, \xi_i]$. Hence $d(y, v)={\rm lim}^{\omega}d_{\mathbb{H}}(y_i, v_i)/\mu_i\leq {\rm lim}^{\omega}d_{\mathbb{H}}(\epsilon_i, v_i)/\mu_i=d(\epsilon, v)$. This implies $B(y)\leq B(\epsilon)=s$. Hence $y\in B^{-1}((-\infty, s])$.

Suppose $y=[\{y_i\}]\in B^{-1}((-\infty, s])$. 
Let $v=[\{v_i\}]$ be a point in $[\epsilon, \xi]\cap[y, \xi]$ where $v_i$ is chosen to be on the geodesic (geodesic ray) $[\epsilon_i, \xi_i]$. Then we have $d(y, v)\leq d(\epsilon, v)$ since $y\in B^{-1}((-\infty, s])$. Hence we have ${\rm lim}^{\omega}d_{\mathbb{H}}(y_i, v_i)/\mu_i\leq {\rm lim}^{\omega}d_{\mathbb{H}}(\epsilon_i, v_i)/\mu_i$. Hence there exists $\{l_j\}$ with ${\rm lim}_{j\rightarrow \infty}l_j=0$, such that for each $j$ we have $d_{\mathbb{H}}(y_i, v_i)/\mu_i\leq d_{\mathbb{H}}(\epsilon_i, v_i)/\mu_i+l_j$ \walmost surely. Let $y^j_i$ be the point in geodesic $[y_i, v_i]$ that is $l_j\mu_i$ away from $y_i$. Then we $d_{\mathbb{H}}(y^j_i, v_i)=d_{\mathbb{H}}(y_i, v_i)-l_j\mu_i\leq d_{\mathbb{H}}(\epsilon_i, v_i)$. This implies $y^j_i\in B_i(h)$.  Hence $h$ fixes $y^j=[\{y^j_i\}]$ by Lemma \ref{efl}. Now apply the same argument as in second paragraph of the proof of Lemma \ref{ball} we see that  $y\in {\rm Fix}_{\mathcal{X}}(h)$. 
\end{proof}

\begin{lemma}\label{fixedpowerrotation}
Let $\gamma_i\in {\rm PSL}(2, \mathbb{R})$ an elliptic (parabolic) element for each $i$ and suppose the order of $\gamma_i$ goes to infinity as $i\rightarrow \infty$. Let $s, t>0$. Suppose  \walmost surely we have $f_i(h)=\gamma_i^s$ and $f_i(h')=\gamma_i^t$. Then ${\rm Fix}_{\mathcal{X}}(h)={\rm Fix}_{\mathcal{X}}(h')$.
\end{lemma}

\begin{proof}
We restrict $i$ to a subset $P$ of $\mathbb{N}$ such that $\omega(P)=1$ and the assumption of the lemma is true whenever $i\in P$. Let $R_i$ be the $\delta$-radius of $\gamma_i$.  Let $\epsilon_i$, $\epsilon'_i$ and $\epsilon^0_i$ be the intersection of a geodesic ray (line) from $\xi_i$ with $\partial B_i(h)$,  $\partial B_i(h')$ and the boundary of the $\delta$-ball (horoball) of $\gamma_i$, respectively. Then $d_{\mathbb{H}}(\epsilon_i, f_i(h)\epsilon_i)=\delta$ and $d_{\mathbb{H}}(\epsilon^0_i, f_i(h)\epsilon^0_i)\leq sd_{\mathbb{H}}(\epsilon^0_i, \gamma_i\epsilon^0_i)=s\delta$. So by Gromov hyperbolicity of $\mathbb{H}$, we have $d_{\mathbb{H}}(\epsilon_i, \epsilon^0_i)=\frac{1}{2}d_{\mathbb{H}}(\epsilon^0_i, f_i(h)\epsilon^0_i)\leq \frac{1}{2}s\delta+4\delta$. Similarly, we have $d_{\mathbb{H}}(\epsilon'_i, \epsilon^0_i)\leq \frac{1}{2}t\delta+4\delta$. So $d_{\mathbb{H}}(\epsilon_i, \epsilon'_i)\leq \frac{1}{2}(s+t)\delta$. Hence $[\{\epsilon'_i\}]=[\{\epsilon_i\}]$. Note that $[\{\epsilon'_i\}]$ and $[\{\epsilon_i\}]$ are points in boundaries of ${\rm Fix}_{\mathcal{X}}(h')$ and ${\rm Fix}_{\mathcal{X}}(h)$, respectively. Since ${\rm Fix}_{\mathcal{X}}(h)$ and ${\rm Fix}_{\mathcal{X}}(h')$ are both balls (horoball) with the same center, we have ${\rm Fix}_{\mathcal{X}}(h)={\rm Fix}_{\mathcal{X}}(h')$.
\end{proof}

\begin{lemma}\label{nondearcfixed}
Let $h, h'\in H$ and suppose ${\rm Fix}_{\mathcal{X}}(h')$ is a proper subset of ${\rm Fix}_{\mathcal{X}}(h)$. Then there exists a non-degenerate segment in the minimal $G$-invariant subtree $\mathcal{T}$ of $\mathcal{X}$ fixed by $h$ but not $h'$.
\end{lemma}

\begin{proof}
Note that $\mathcal{T}$ is not contained in ${\rm Fix}_{\mathcal{X}}(h)$ since otherwise $h$ is in the kernel of the action of $G$ on $\mathcal{T}$. Let $z=[\{z_i\}]\in \mathcal{T}-{\rm Fix}_{\mathcal{X}}(h)$. Then by Lemma \ref{efl}, $z_i$ is not in $B_i(h)$ \walmost surely. Let $\epsilon_i$ and $\epsilon'_i$ be the intersection of the segment $[\xi_i, z_i]$ with $\partial B_i(h)$ and $\partial B_i(h')$, respectively. Then $\epsilon=[\{\epsilon_i\}]$ and $\epsilon'=[\{\epsilon'_i\}]$ are points in ${\rm Fix}_{\mathcal{X}}(h)$ and ${\rm Fix}_{\mathcal{X}}(h')$ closest to $z$, respectively. Since $h\in \partial {\rm Fix}_{\mathcal{X}}(h)$ and ${\rm Fix}_{\mathcal{X}}(h)$ contains ${\rm Fix}_{\mathcal{X}}(h')$ properly, we know that $\epsilon\notin {\rm Fix}_{\mathcal{X}}(h')$. Hence $[\epsilon, \epsilon']$ is fixed by $h$ but not by $h'$. To finish the proof, we now show that $[\epsilon, \epsilon']\subset \mathcal{T}$.  Since $\epsilon$ is on the segment $[z, \epsilon']$ and $z\in \mathcal{T}$, it suffices to show that $\epsilon'\in\mathcal{T}$. Consider the geodesic tripod $[z, h' z, \epsilon']$. Since $h'$ fixes $\epsilon'$, it also fixes $y$, the center of the $[z, h' z, \epsilon']$. Clearly $y$ is closer to $z$ than $\epsilon'$ is unless $y=\epsilon'$. Since $\epsilon'$ is the point in ${\rm Fix}_{\mathcal{X}}(h')$ closest  to $z$. So $y=\epsilon$ and $[z, h' z, \epsilon']$ is actually the geodesic from $z$ to $h'z$ with $\epsilon'$ as the mid-point. Since $\mathcal{T}$ is $G$-invariant, we know that $[z, h'z]$ and hence $\epsilon'$ is in $\mathcal{T}$. 
\end{proof}

\begin{lemma}\label{Ell and para has fixed point}
Suppose none of the non-degenerate segment in the minimal $G$-invariant subtree $\mc{T}$ of $\mathcal{X}$ has elliptic or parabolic stabilizer. Then every elliptic or parabolic subgroup $H$ of $G$ has a (unique) fixed point in $\mc{T}$. 
\end{lemma}

\begin{proof}
By the assumption of the lemma and Lemma \ref{nondearcfixed}, we have ${\rm Fix}_{\mathcal{X}}(h')={\rm Fix}_{\mathcal{X}}(h)$ for any $h, h'\in H$. Let $D={\rm Fix}_{\mathcal{X}}(h)\subset \mc{X}$. Then $D$ is fixed pointwise by every element of $H$ and every point outside of $D$ is not fixed by any element of $H$. By the assumption of the lemma, $D\cap \mc{T}$ contains at most one point.  We now show that $\mc{T}\cap D$ contains a single point. Let $x\in \mc{T}$ and $h\in H$. Let $d\in D$ be the point in $D$ closest to $x$. Then $h$ fixes the center $o$ of the tripod with end points $x$, $hx$, and $d$. So $o=d\in D$. On the other hand, since $\mc{T}$ is $G$-invariant, it contains the geodesic $[x, hx]$ and hence also $o$. Therefore, $\mc{T}\cap D=o$.    
\end{proof}

When none of the non-degenerate segment in the minimal $G$-invariant subtree $\mc{T}$ of $\mathcal{X}$ has elliptic or parabolic stabilizer, the Rips machine \cite[Theorem 5.1]{Guirardel} can be applied to the action of $G$ on $T$ and we can run the shortening argument. (See Section \ref{hs} for the details. )

In the case where there is a non-degenerate segment in $\mc{T}$ fixed by an elliptic (or parabolic ) subgroup $H$, we use a construct by Groves, Manning and Wilton to ``separate'' $H$ from $G$.

Let $J\subset \mathcal{T}$ be a non-degenerate segment whose non-trivial stabilizer $H_J\subset L$ is elliptic (parabolic). Let $H$ be the maximal elliptic (parabolic) subgroup of $G$ containing $H_J$. Let $\xi_i$ be the fixed point of $f_i(H)$ and let $\xi=[\{\xi_i\}]\in\mc{X}\cup\partial \mc{X}$.

Define a subtree $Y$ of $\mathcal{X}$ by $Y=\text{The closure of }\cup_{h\in H}{\rm Fix}_{\mathcal{X}}(h)$. We call $Y$ the {\em fixed ball} of $H$ in $\mathcal{X}$. Note that $Y$ intersect $\mathcal{T}$ in a non-degenerate subtree since both $Y$ and $\mathcal{T}$ contain $J$.  Also note that $Y$ is a ball ($\xi\in\mathcal{X}$) or a horoball ($\xi\in\partial\mathcal{X}$) centered at $\xi$ since ${\rm Fix}_{\mathcal{X}}(h)$ is a ball (or a horoball) for all $h\in H$.

\begin{lemma}\label{intersectionapoint}
If $g\in G\backslash H$, then $gY$ and $Y$ intersect at at most one point. In particular, if $y\in gY\cap Y$, then $y\in \partial Y$. 
\end{lemma}

\begin{proof}
Suppose the lemma is not true. Let $Q$ be a non-degenerate segment in $gY\cap Y$. Since $Q\subset gY$, we have $g^{-1}Q\subset Y$. Let $m={\rm sup}\{B(y)\mid y\in Q\cup g^{-1}Q\}$. Since  $Q\cup g^{-1}Q\subset Y$, we have $m\leq s$. In fact if we replace $Q$ by a non-degenerate subsegment of itself, we can assume $m< s$. Thus, there exists nontrivial $h\in H$ that fixes $Q\cup g^{-1}Q$. Then for any $q\in Q$ we have $[h, g]q=hgh^{-1}g^{-1}q=hgg^{-1}q=hq=q$. Hence $[g,h]$ fixes $Q$. Therefore by Lemma \ref{arc sta abel} we have $[[g,h], h]=1$, which implies $[g,h]\in H$. Hence $ghg^{-1}\in H$. This implies that \walmost surely $f_i(ghg^{-1})$ and $f_i(h)$ is in the same elliptic (parabolic) subgroup of $\Gamma_i$. Therefore $f_i(g)$ must be in the same elliptic (parabolic) subgroup of $\Gamma_i$ \walmost surely, which implies that   $g\in H$. Contradiction.  
\end{proof}

\begin{corollary}\label{preserveY}
If $g\in G$ preserves $Y$, then $g\in H$. 
\end{corollary}


A {\em star} is the complete bipartite graph $K_{1,n}$. The vertex connected to all the others is called the {\em center vertex}. In the following theorem, we allow multiple edges between the center vertex and the other vertices. 

\begin{theorem}[Groves-Manning-Wilton]\label{GMWtree}
Assume \ref{standing assumption}. Suppose $G$ is not abelian and is finitely generated relative to a collection of elliptic or parabolic subgroups $\{H_1, \dots, H_k\}$.  Let $S$ be a finite generating set of $G$ relative to $\{H_1, \dots, H_k\}$. Suppose 
\begin{enumerate}
\item the $\omega$-kernel of $\{f_i\}$ is trivial. 
\item $|f_i|_S\rightarrow \infty$.  
\item $\{f_i\}$ induces an action of $G$ on the asymptotic cone $\mc{X}$ of $\mathbb{H}$ defined by $\{|f_i|_S\}$ and centrally located points $\{x_i\}$. 
\item The minimal $G$-invariant subtree $\mc{T}$ of $\mc{X}$ contains a non-degenerate segment with elliptic (or parabolic) stabilizer $H'\subset G$. 
\end{enumerate}
Then $G$ admits a non-trivial splitting $\mathbb{G}$ with the following properties:
\begin{enumerate} 
\item The underlying graph of $\mathbb{G}$ is a star;
\item The center vertex group $H$ is the maximal abelian group containing $H'$; 
\item All the edge groups are conjugates of $H_0$, where $H_0$ is the point-wise stabilizer of the fixed ball $Y$ of $H$ in $\mathcal{X}$. 
\item If $H_0\neq \{1\}$, each edge group is maximal abelian in the non-center vertex group containing it.
\item If a subgroup $G_0$ of $G$ fixes a point $x\in \mc{X}$, then it is conjugate to a subgroup of a vertex group of $\mathbb{G}$.
\item If $E$ is a parabolic or elliptic subgroup of $G$, then $E$ is conjugate to a subgroup of a vertex group of $\mathbb{G}$.
\item Let $E$ be a parabolic or elliptic subgroup not conjugating to a subgroup of $H$. If a non-trivial subgroup subgroup of $E$ and a subgroup $G_0$ of $G$ both fix the same point $x\in \mc{X}$, then $E$ and $G_0$ are conjugate to the same non-center vertex group of $\mathbb{G}$ by a common element of $G$. 
\end{enumerate}
\end{theorem}

A splitting $\mathbb{G}$ is called an {\em elliptic (or parabolic) splitting} of $G$ relative to $\{H_1, \dots, H_k\}$ induced by $\{f_i\}$ if it is obtained by applying the above theorem. 

\begin{proof}
Let $Y$ be the fixed ball of $H$. Then $Y$ contains the non-degenerate segment fixed by $H'$. Let $\mathcal{T}_Y=\mathcal{T}\cup (\bigcup_{g\in G}gY)$. This is the smallest $G$-invariant subtree of $\mathcal{X}$ containing $Y$.  We define a simplicial tree $T_{Y}$ as follow:
There are two types of vertices:  
\begin{enumerate}
\item Vertices corresponding to translates of $Y$.
\item Vertices corresponding to components of $\mathcal{T}_Y-(\bigcup_{g\in G}int(gY))$, where $int(gY)$ is the topological interior of $gY$ as a subspace of $\mathcal{X}$. 
\end{enumerate}

Edges: There is an edge between a vertex of the first type and one of the second type if and only if they (the corresponding subsets of $\mc{X}$) intersect. There is no edge between vertices of the same type. 

It is clear from the definition that $T_{Y}$ is a bipartite graph, with one color being the first type of vertices and the other color being the second type. Note that $T_{Y}$ has no loop with two edges since there is at most one edge between any two points in $T_Y$. Any loop in $T_{Y}$ with more than two edges defines a non-trivial loop in $\mathcal{T}$, which can not exist as $\mathcal{T}$ is an $\mathbb{R}$-tree. So $T_{Y}$ is a tree. One can check that the action of $G$ on $\mathcal{T}_Y$ induces an action of $G$ on $T_{Y}$, which preserves the types of vertices. As a result, the $G$-action on $T_{Y}$ has no inversion as the types of the two vertices of an edge are different.

We now show that $G$ does not have a global fixed point in $T_{Y}$. Suppose there is point $v\in T_{Y}$ fixed by $G$. Then we can assume $v$ is a vertex since the $G$-action has no inversion. Suppose $v$ is a first type vertex. Then $G$ preserve $gY$ for some $g\in G$. Hence by Corollary \ref{preserveY}, $G=gHg^{-1}$. This is impossible as $H$ is abelian and $G$ is not. Now suppose $v$ is a second type vertex. Then $G$ preserves the proper subtree $T$ of $\mathcal{T}_Y$ corresponding to $v$. Hence $T$ is a $G$-invariant subtree of $\mathcal{X}$ and hence $T$ contains $\mathcal{T}$. By the definition of $T_{Y}$ the intersection of $Y$ and  $T$ contains at most one point.  However by definition of $Y$, the intersection of $Y$ and $\mathcal{T}$ is a non-degenerate subtree. Therefore $T$ can not contain $\mc{T}$. Contradiction. 

Let $\mathbb{G}$ be the graph of groups decomposition corresponding to the action of $G$ on $T_Y$. Since the action of $G$ on $T_Y$ is non-trivial, $\mathbb{G}$ is non-trivial. Note that all type-(1) vertices in $T_Y$ are in the same $G$-orbit since they correspond to $G$-translates of $Y$. Denote the vertex in $\mathbb{G}$ corresponding to this $G$-orbit by $v$ and without loss of generality we can assume that the vertex group at $v$ is the set-wise stabilizer of $Y$, which is $H$. Since $T_Y$ is a connected bipartite graph, the quotient of the type-(2) vertices are all connected to $v$ and have no edges between each other. So the underlying graph of $\mathbb{G}$ is a star and $v$ is the center vertex. Hence Property (1) and (2) are proved. 

Let $e$ be an edge of $T_Y$ and let $H_e$ be the stabilizer of $e$ in $G$.  Then $H_e$ preserves $gY$ for some $g\in G$ and a component $T_e$ of $\mathcal{T}_Y-(\bigcup_{g\in G}int(gY))$. Hence $H_e$ fixes the point $gY$ and $T_e$ intersect at.  By Corollary \ref{preserveY} ${\rm stab}(e)\subset gHg^{-1}$. Therefore the fixed point set of $H_e$ is a ball (or a horoball) with the same center as $gY$. Since $H_e$ fixes the point of intersection of $T_e$ and $Y$, which is in the topological boundary of $gY$, we know that $H_e$ fixes $gY$. So $H_e\subset gH_0g^{-1}$. On the other hand, $gH_0g^{-1}$ fixes $gY$ and hence it preserves the component corresponding to the type-(2) vertex of $e$. Therefore,  $gH_0g^{-1}\subset H_e$. So $gH_0g^{-1}= H_e$ and Property (3) is proved. 

Since $H$ is a maximal abelian group in $G$ and $G$ is commutative translative by Lemma \ref{l:comm tran},  $H$ contains all abelian subgroups containing $H_e\neq \{1\}$. Hence $H_e$ has to be maximal abelian in the other vertex group containing it. Property (4) follows. 

We now prove (5). $G_0$ fixes the point $x'$ in $\mathcal{T}_Y$ closest to $x$. If $x'$ is in the interior of some translate of $Y$, then $G_0$ fixes the vertex in $T_Y$ corresponding to that translate of $Y$. Otherwise, there is connected component of $\mathcal{T}_Y-(\bigcup_{g\in G}int(gY))$ that contains $x'$. In this case, $G_0$ fixes the vertex in $T_Y$ corresponding to this component. Hence $G_0$ is conjugate to a subgroup of a vertex group of $\mathbb{G}$. 

Let $E$ be a parabolic (or elliptic) subgroup of $G$. Suppose $E$ is not conjugate to a subgroup of $H$. (Otherwise the  (6) is obviously true. )  Let $E_1\subset E_2\subset E$ be non-trivial finitely generated subgroups. Then $E_1$ fixes a point $x\in \mc{X}$ since $E$ is parabolic (or elliptic). Hence $E_1$ fixes a vertex $v$ of $T_Y$ by the previous paragraph.  Similarly, $E_2$ fixes a vertex $u$ of $T_Y$. Suppose $u\neq v$. Since $E_1$ fixes both $u$ and $v$, it fixes a vertex corresponding to a translate of $Y$. Hence $E_1$ is conjugate to a subgroup of $H$. Since $H$ is maximal abelian, $E$ is abelian and $E_1$ is non-trivial, we know that $E$ is conjugate to a subgroup of $H$. Contradiction. So $u=v$. Hence all finitely generated subgroups of $E$ fix $v$. Then Property (5), $H$ is conjugate to a subgroup of a vertex group of $\mathbb{G}$. 

We now prove (7). Let $E'$ be the non-trivial subgroup of $E$ fixing $x$. By the argument of in the proof of (5), the fact that both $E'$ and $G_0$ fix $x\in \mc{X}$ implies that both $E'$ and $G_0$ fix the same vertex $v$ of $T_Y$. Note that $v$ is a type (2) vertex since $E'$ is not conjugate to a subgroup of $Y$. By the argument in the proof of (6), $E$ also fixes $v$ since $E$ is not conjugate to a subgroup of $H$. (7) follows. 
\end{proof}

We now take a closer look at the edge groups and the center vertex group of elliptic (and parabolic) splittings. Note that each of them is either elliptic or parabolic.  We derive some properties of them that we need.  

\begin{definition}
We call an elliptic subgroup $H\subset G$ a {\em small elliptic subgroup} if for any $1\neq h\in H$ the \w limit of the angle of the rotation of $f_i(h)$ is zero.
\end{definition}

We omit the proof of the following easy lemma. 
\begin{lemma}
Small elliptic subgroups are torsion free.
\end{lemma}

\begin{lemma}\label{smallrotation}
Let $\mathbb{G}$ be an elliptic splitting. Then the edge group $H_0$ of $\mathbb{G}$ is small elliptic. In particular, $H_0$ is torsion free.
\end{lemma}

\begin{proof}
Suppose the lemma is not true. Then for some $h\in H_0$, there exists $\theta>0$ such that the angle of the rotation $f_i(h)$ is greater than $\theta$ \walmost surely. Since $h\in H_0$, we know that $h$ fixes $Y$. Let $y=[\{y_i\}]\in\partial Y$. Then the \w limit of $d_i(f_i(h)y_i, y_i)=\frac{1}{\mu_i}d_{\mathbb{H}}(f_i(h)y_i, y_i)$ is zero. Let $\xi_i$ be the fixed point of $f_i(H_0)$. Since $Y$ has positive diameter, we know that the \w llimit of $d_i(y_i, \xi_i)=\frac{1}{\mu_i}d_{\mathbb{H}}(y_i, \xi_i)$ is greater than $c$ for some $c>0$ ($c$ could be infinity). Since the angle of rotation of $f_i(h)$ is greater than a positive number $\theta$, we have $d_{\mathbb{H}}(f_i(h)(y_i), y_i)\geq d_{\mathbb{H}}(y_i, \xi_i)$ whenever $d_{\mathbb{H}}(y_i, \xi_i)$ is big enough. Therefore we can not have ${\rm lim}^{\omega}\frac{1}{\mu_i}d_{\mathbb{H}}(f_i(h)y_i, y_i)=0$ while ${\rm lim}^{\omega}\frac{1}{\mu_i}d_{\mathbb{H}}(y_i, \xi_i)=c$ for some $c>0$. Contradiction. 
\end{proof}

\begin{lemma}\label{torsionfree}
Let $\mathbb{G}$ be an elliptic splitting with edge group $H_0$ and center vertex group $H$. Suppose $H$ is small elliptic. Then $H/H_0$ is torsion free. 
\end{lemma}

\begin{proof}
Suppose $H/H_0$ has torsion. Then there exists $h\in H-H_0$ such that $h^d\in H_0$ for some integer $d$. Since $h\notin H_0$, there exists $y\in Y$, such that $hy\neq y$. By assumption of the lemma, the angle of the rotation $f_i(h)$ goes to zero as $i$ goes to infinity over a full $\omega$-measure set. As a result the angle of the rotation $f_i(h^d)$ is greater than the angle of the rotation $f_i(h)$ \walmost surely. Hence $h^dy\neq y$. This contradicts the fact that $h^d\in H_0$.
\end{proof}

\begin{lemma}\label{p-torsionfree}
Let $\mathbb{G}$ be a parabolic splitting with edge group $H_0$ and center vertex group $H$. Then $H/H_0$ is torsion free. 
\end{lemma}

\begin{proof}
Suppose $H/H_0$ has torsion. Then there exists $h\in H-H_0$ such that $h^d\in H_0$ for some integer $d$. Since $h\notin H_0$, there exists $y\in Y$, such that $hy\neq y$. For any $y_i\in \mathbb{H}$, $f_i(h^d)$ moves $y_i$ further than $f_i(h)$ does. Hence $hy\neq y$ implies $h^dy\neq y$. This contradicts the fact that $h^d\in H_0$.
\end{proof}

\section{\large Good relative generating set}\label{s:rf}

In this section, we assume \ref{standing assumption} and \ref{SA2}.

 As explained in Section \ref{s:sandp}, in order to prove Theorem \ref{MT2}, we need to run the shortening argument on non-abelian vertex groups of elliptic (or parabolic) splittings of a finitely generated group $G$. However, these vertex groups are only know to be finitely generated relative to the adjacent edge groups and are not known to be finitely generated. Since translation length is only well defined with respect to a finite subset, not having a finite generating set causes the following problems: Let $G_v$ be a non-abelian vertex groups of elliptic (or parabolic) splittings of $G$ and $S$ be any finite subset of $G_v$.
\begin{enumerate}
\item  If ${\rm lim}^\omega|f_i|_S=\infty$, $\{f_i\}$ may not induce an action of $G_v$ on the asymptotic cone of $\mathbb{H}$ defined by $\frac{1}{|f_i|_S}d_{\mathbb{H}}$. However such an action is a key ingredient in the shortening argument;
\item If ${\rm lim}^\omega|f_i|_S<\infty$, one may still find another finite subset $S'\subset G_v$ such that  $\{|f_i|_{S'}\}$ is unbounded. Hence we still can not run the argument in Section \ref{s:Bounded} to produce a factoring map for $f_i|_{G_v}$. 
\end{enumerate} 
The goal of this section is to prove Proposition \ref{relativewelldefined}, which deals with the above problems. First we set up the notations. 
   
Suppose a countable group $G$ is finitely generated relative to a finite collection of pair-wise non-conjugate maximal elliptic or parabolic subgroups $\mathcal{H}=\{H^1, \dots, H^k\}$. Suppose further that if $H^j$ is elliptic then it is small elliptic. Let $S_0$ be a finite subset of $G$. Let $\mu_i=|f_i|_{S_0}$ and $d_i=\frac{1}{\mu_i}d_{\mathbb{H}}$. Let $x_i$ be a centrally located point in $\mathbb{H}$ of $f_i$ with respect to $S_0$ (Lemma \ref{l:cent located points exits}).  We call $S_0$ a {\em good generating set of $G$ relative to $\{H^1, \dots, H^k\}$} with respect to $\{f_i\}$ if the following are true:
\begin{enumerate}
\item $S_0$ intersects each $H^j$ non-trivially.
\item $S_0$ generates $G$ together with $\{H^1, \dots, H^k\}$. 
\item For each $h\in H^j$, we have that $d_i(x_i, f_i(h)x_i)$ has finitely $\omega$-limit.
\end{enumerate}
Having a good generating set $S_0$ solves the two problems pointed out in the first paragraph: If $\{\mu_i\}$ is bounded, then $|f_i|_{S}$ is \w bounded for any finite subset $S$ of $G$. If $\{\mu_i\}$ is unbounded, then $G$ admits a well defined action induced by $\{f_i\}$ on the asymptotic cone $\mathcal{X}$ of $\mathbb{H}$ defined by $\{x_i\}$,$\{\mu_i\}$ and $\omega$. 

While finite subsets satisfying the first two conditions of good generating set are easy to find, one might not be able to find one satisfying the last condition. Let $S_0$ be a finite subset satisfying the first two conditions of good generating set. Let $H^j_0$ be the subset of $H^j$ consisting of elements $h$ such that $d_i(x_i, f_i(h)x_i)$ is \w bounded. Note that $H^j_0$ is a subgroup. Let $G_0$ be the subgroup of $G$ generated by $S_0$ and $\mathcal{H}_0=\{H^1_0, \dots, H^k_0\}$.  Note that $S_0$ is a good generating set for $G_0$ relative to $\mathcal{H}_0$. The following proposition says that having a good generating set for $G_0$ is ``good enough'': 

\begin{proposition}\label{relativewelldefined}
With the above notation and assumption,
$G=G_0*_{(H_0^1,\dots, H_0^k)}(H^1, \dots, H^k)$.  
\end{proposition}

We prove the above lemma in the rest of this section. We need the following two technical lemmas. 

\begin{lemma}\label{d}
With the notation above, let $R\subset \cup_{j}(H^j-H^j_0)$ be finite. Let $\nu_i$ be the translation length of $f_i$ with respect to $S_R=S_0\cup R$. Let $\mu_i=|f_i|_{S_0}$ and $x_i$ be a centrally located point in $\mathbb{H}$ of $f_i$ with respect to $S_0$. Then there is a centrally located point $y_i$ of $f_i$ with respect to $S_R$ such that: 
\begin{enumerate}
\item ${\rm lim}^{\omega}\nu_i=\infty$
\item ${\rm lim}^{\omega}d_{\mathbb{H}}(x_i, y_i)/\nu_i<\infty$. 
\item ${\rm lim}^{\omega}\mu_i/\nu_i= 0$.
\end{enumerate}
\end{lemma}

\begin{proof}
We first show that ${\rm lim}^{\omega}\nu_i=\infty$. Since $S_0\subset S_R$, we have $\nu_i\geq\mu_i$.  If ${\rm lim}^{\omega}\mu_i=\infty$, then ${\rm lim}^{\omega}\nu_i=\infty$. Now consider the case when $\mu_i$ has finite $\omega$-limit. Suppose $\nu_i$ also has finite $\omega$-limit. Let $H_i$ be the maximal elliptic (or parabolic) subgroup of $\Gamma_i$ containing $f_i(H^1)$. Let $D_i=D(H_i)$ be the Margulis domain  associated to $H_i$ (See Section \ref{s:Bounded}).  Since $S_0$ contains an element inside of $H^1$ and also an element outside of $H^1$, we know from Lemma \ref{boudnary disct} that $d_{\mathbb{H}}(x_i, \partial D_i)$ has finite $\omega$-limit.  

Since we assume that $\nu_i$ has finite $\omega$-limit, the same argument as above shows that $d_{\mathbb{H}}(y_i, \partial D_i)$ also has finite $\omega$-limit. As a result, $|d_{\mathbb{H}}(x_i,  f_i(h')x_i)-d_{\mathbb{H}}(y_i,  f_i(h')y_i)|$ has finite $\omega$-limit for any $h'\in R$. Since $d_{\mathbb{H}}(y_i,  f_i(h')y_i)\leq \nu_i$ has finite $\omega$-limit, $d_{\mathbb{H}}(x_i,  f_i(h')x_i)$ also has finite $\omega$-limit. This contradicts the definition of $h'$. So ${\rm lim}^{\omega}\nu_i=\infty$. 

We now prove (2) . Suppose (2) is not true. We claim that for some $h_0\in R$ $y_i$ is closer to the fixed point of $f_i(h_0)$ than $x_i$ is. (Proof of the claim: If $y_i$ is NOT closer to the fixed point of $f_i(h)$ than $x_i$ is for all $h\in R$, then we have $d_{\mathbb{H}}(x_i, f_i(h)x_i)\leq d_{\mathbb{H}}(y_i, f_i(h)y_i)$ for all $h\in R$.  Then $x_i$ could have been chosen as a centrally located point with respect to $S_R$ and (1) follows.) We consider the case where $h_0$ is elliptic. Let $\xi_i$ denote the fixed point of $f_i(h_0)$ in $\mathbb{H}$. 

 Let $\mathcal{X}_R$ be the asymptotic cone of the hyperbolic plane $\mathbb{H}$ defined by $\{y_i\}$, $\{\nu_i\}$ and $\omega$. Denote its metric by $d^R$. Let $y=[\{y_i\}]$. Since ${\rm lim}^{\omega}d_{\mathbb{H}}(x_i, y_i)/\nu_i=\infty$, we know that $x=[\{x_i\}]$ is in the boundary of $\mathcal{X}_R$. Denote the geodesic ray from $y$ to $x$ by $[y, x]$.  Let $h\in S_0$ such that $h$ and $h_0$ are in the same elliptic subgroup of $G$. Let $z_i$ be any point in the geodesic segment $[x_i, y_i]$, then $z_i$ is closer to $\xi_i$ than $x_i$. Since $f_i(h)$ is elliptic(parabolic) with fixed point $\xi_i$, we have $d_{\mathbb{H}}(hz_i, z_i)\leq d_{\mathbb{H}}(hx_i, x_i)\leq \mu_i\leq \nu_i$. Suppose $z=[\{z_i\}]$ is in $\mc{X}_R$.  We have $d^R(hz, z)={\rm lim}^{\omega}d_{\mathbb{H}}(hz_i, z_i)/\nu_i \leq1$. So $h$ moves the ray from $y$ to $x$ by a bounded amount. This implies that $h$ fixes $x$. On the other hand, $h$ fixes  $\xi=[\{\xi_i\}]\in \mathcal{X}_R\cup\partial \mathcal{X}_R$. Hence by Lemma \ref{ball} (or Lemma \ref{horoball}) it fixes all points in $\mathcal{X}_R$. Therefore $h$ is in the $\omega$-kernel of $\{f_i\}$, contradicting the assumption at the beginning of the section. The case where $h_0$ is parabolic is similar. So (2) is proved. 

Now we prove (3). Let $h_0\in R$ and Let $l_i=d_{\mathbb{H}}(h_0x_i, x_i)$. Since $h_0\in R$, we know that ${\rm lim}^{\omega}\mu_i/l_i= 0$. Consider the geodesic rectangle $[x_i, y_i, h_0y_i, h_0x_i]$. By triangle inequality, we have 
\begin{eqnarray*}
l_i=d_{\mathbb{H}}(h_0x_i, x_i)&\leq& d_{\mathbb{H}}(x_i, y_i)+d_{\mathbb{H}}(y_i, h_0y_i)+d_{\mathbb{H}}(h_0x_i, h_0y_i)\\ &\leq& 2d_{\mathbb{H}}(x_i, y_i)+d_{\mathbb{H}}(y_i, h_0y_i)
\end{eqnarray*}
By (1), we have $d_{\mathbb{H}}(x_i, y_i)\leq D\nu_i$ for some $D$ $\omega$-almost surely. By definition of $\nu_i$, we have $d_{\mathbb{H}}(y_i, h_0y_i)\leq \nu_i$. Hence we have: 
\begin{eqnarray*}
l_i&\leq& 2d_{\mathbb{H}}(x_i, y_i)+d_{\mathbb{H}}(y_i, h_0y_i)\\ &\leq&  (2D+1)\nu_i
\end{eqnarray*}
Therefore $\omega$-almost surely 
\begin{equation*}
\frac{\mu_i}{\nu_i}\leq \frac{(2D+1)\mu_i}{l_i},
\end{equation*}
which implies (3).
\end{proof}


\begin{lemma}\label{fgswd}
With the notation above, let $R\subset  \cup_{j}(H^j-H^j_0)$ be finite. Let $\bar G$ be the subgroup of $G$ generated by $S_R=S_0\cup R$. Let $\bar H^j=\bar G\cap H^j$. Then $\bar G$ is isomorphic to $G_0$ amalgamated with $\bar H^j$ along $H^j_0$ for $j=1, \dots, k$.   
\end{lemma}

\begin{proof}
The plan is to obtain the splitting we want by applying Theorem \ref{GMWtree} multiple times. 

Let $\nu_i$ be the translation length of $f_i$ with respect to $S_R$ and let $y_i$ be a corresponding centrally located point. By (1) of Lemma \ref{d} ${\rm lim}^{\omega}\nu_i=\infty$. Hence the asymptotic cone $\mathcal{X}_R$ of $\mathbb{H}$ defined by $\{y_i\}$, $\{\nu_i\}$ and $\omega$ is an $\mathbb{R}$-tree. Let $d^R$ be the metric of $\mathcal{X}_R$. Since $S_R$ generates $\bar G$, $\{f_i\}$ induces an action of $\bar G$ on $\mc{X}_R$. By (1) of Lemma \ref{d} we have $x=[\{x_i\}]\in \mc{X}$. Since $S_R$ is finite, there exists $h_0\in S_R$ such that $d_{\mathbb{H}}(x_i, f_i(h_0)x_i)$ is great than $d_{\mathbb{H}}(x_i, f_i(h)x_i)$ for any $h\in S_R$ \walmost surely.    Hence $d_{\mathbb{H}}(x_i, f_i(h_0)x_i)\geq {\rm max}_{h\in S_R}\{d_{\mathbb{H}}(x_i, f_i(h)x_i)\}\geq \nu_i$ \walmost surely. (The second inequality follows from the definition of translation length.) So $d^R(x, h_0x)={\rm lim}^\omega d_{\mathbb{H}}(x_i, f_i(h_0)x_i)/\nu_i\geq 1$. Hence $x\notin {\rm Fix}(h_0)$. Without loss of generality, we assume $h_0\in \bar H^1$. Note that $h_0\in R$. Let $h\in S_0\cap \bar H^1$.  Then we have $d_{\mathbb{H}}(x_i, f_i(h)x_i)\leq \mu_i$. Hence by Lemma \ref{d} we have ${\rm lim}^{\omega}d_{\mathbb{H}}(f_i(h)x_i, x_i)/\nu_i\leq {\rm lim}^{\omega}\mu_i/\nu_i=0$. So $d^R(x, hx)=0$. Hence $x\in {\rm Fix}_{\mathcal{X}}(h)$.  So $ {\rm Fix}_{\mc{X}}(h_0)$ is a proper subset of ${\rm Fix}_{\mathcal{X}}(h)$. Therefore, by Lemme \ref{nondearcfixed}, $h$ fixes a non-degenerate segment in $\mathcal{T}$, the minimal $\bar G$-invariant subtree of $\mathcal{X}_R$. Therefore Theorem \ref{GMWtree} applies and yields an elliptic (or parabolic) splitting $\mathbb{G}$ of $\bar G$. Since $H^1$ is a maximal abelian subgroup of $G$, $\bar H^1$ is a maximal abelian subgroup of $\bar G$. Since $h\in \bar H^1$, we know that the center vertex group of $\mathbb{G}$ is $\bar H^1$. Let $E_1$ be the edge group of $\mathbb{G}$. By assumption if $\bar H^1$ is elliptic, then $\bar H^1$ is a small elliptic subgroup. Hence by Lemma \ref{torsionfree}, $\bar H^1/E_1$ is torsion free. If $\bar H^1$ is parabolic, $\bar H^1/E_1$ is also torsion free by Lemma \ref{p-torsionfree}. Since $\bar G$ is finitely generated, $\bar H^1/E_1$ is finitely generated and hence a free abelian group. Note that the rank of $\bar H^1/E_1$ is at least one since $h_0\notin E_1$.  Note also that $\mathbb{G}$ induces an epimorphism from $\bar G$ to the free abelian group $\bar H^1/E_1$ by collapsing all the vertex groups of $\mathbb{G}$ except the center vertex group. 

Our next goal is to show that $G_0,\bar H^2, \dots, \bar H^k$ are all in (a conjugate of) a non-abelian vertex group of $\mathbb{G}$. By (7) of Theorem \ref{GMWtree}, it suffices to show that $G_0$ fixes a point in $\mc{X}_R$ and each of $\bar H^2, \dots, \bar H^k$ has a non-trivial subgroup fixing the same point. By definition of $G_0$, for every $g\in G_0$, we have ${\rm lim}^\omega d_\mathbb{H}(f_i(g)x_i, x_i)/\mu_i<\infty$. This means that \walmost surely $d_{\mathbb{H}}(f_i(g)x_i, x_i)\leq D\mu_i$ for some $D$ independent of $i$. Hence we have 
\begin{equation*}
d_{\mathbb{H}}(f_i(g)x_i, x_i)/\nu_i\leq D\mu_i/\nu_i. 
\end{equation*}
Therefore by (2) of Lemma \ref{d}, we have ${\rm lim}^{\omega}d_{\mathbb{H}}(f_i(g)x_i, x_i)/\nu_i= 0$. So $d^R(gx, x)=0$ for all $g\in G_0$. Hence $G_0$ fixes $x$. Since $\bar H^j$ intersects $G_0$ in a non-trivial subgroup, $\bar H^j$ has a non-trivial subgroup fixing $x$. Hence we know that $G_0,\bar H^2, \dots, \bar H^k$ are all in (a conjugate of) a non-abelian vertex group of $\mathbb{G}$. Since $E_1$ is in all non-abelian vertex groups of $\mathbb{G}$, the group $\bar G^1=<G_0,E_1, \bar H^2, \dots, \bar H^k>$ is in a non-abelian vertex group of $\mathbb{G}$.

Note that $\mathbb{G}$ induces a splitting of $<\bar G^1, \bar H^1>\subset \bar G$ as $\bar G^1*_{E_1}\bar H^1$. Since $\bar G^1$ and $\bar H^1$ generate $\bar G$. We have $\bar G=\bar G^1*_{E_1}\bar H^1$.

As long as $\bar G^1\neq G_0$,  we can repeat the above argument for $(\bar G^1, \{E_1, \bar H^2, \dots, \bar H^k\})$, producing one of the following: 
\begin{enumerate}
\item $\bar G^1=\bar G^2*_{E_j}\bar H^j$, where $\bar G^2=<G_0, E_1, \bar H^2, \dots, \bar H^{j-1}, E_j, \bar H^{j+1}, \dots,  \bar H^k> $and $\bar H^j/E_j$ is free abelian with positive rank. 
\item $\bar G^1=\bar G^2*_{E_1^1}\bar E_1$, where $\bar G^2=<G_0, E_1^1, \bar H^2, \dots, \bar H^k> $ and $\bar E_1/E_1^1$ is free abelian with positive rank. 
\end{enumerate} 
Each of this case we get a refinement $\mathbb{G}_1$ of $\mathbb{G}$:
\begin{enumerate}
\item $\bar G=\bar G^2*_{(E_1, E_j)}(\bar H^1, \bar H^j)$, where $\bar G^2=<G_0, E_1, \bar H^2, \dots, \bar H^{j-1}, E_j, \bar H^{j+1}, \dots,  \bar H^k> $. 
\item $\bar G=\bar G^2*_{E_1^1}\bar H^1$, where $\bar G^2=<G_0, E_1^1, \bar H^2, \dots, \bar H^k> $.
\end{enumerate} 

In either case,  $\mathbb{G}_1$ induces epimorphism from $G$ to a free abelian group (either $\bar H^1/E_1\oplus \bar H^j/E_j$ or $\bar H^1/E_1^1$) with rank strictly greater than $\bar H^1/E_1$. We see that on one hand we can keep refining $\mathbb{G}$ as long as the non-abelian vertex group is not $G_0$ and on the other hand this refining process has to stop after finitely many steps as it produces epimorphims from $\bar G$ to free abelian groups with bigger and bigger ranks. As a result, at the end of the above process, we get the splitting of $\bar G$ as the amalgamated product of $G_0$ with $(\bar H^1, \dots, \bar H^k)$ over $(H^1_0, \dots, H^k_0)$. 
\end{proof}

\begin{proof}[Proof of Proposition \ref{relativewelldefined}]
 Let $\iota: G_0*_{(H_0^1, \dots, H_0^k)}(H^1, \dots, H^k)\rightarrow G$ be the map induced by the inclusions from $G_0$ and $\{H^1, \dots, H^k\}$ to $G$. Since $G_0$ and $\{H^1, \dots, H^k\}$  generate $G$, we know $\iota$ is surjective. By Lemma \ref{fgswd}, we know that for any finite set $R\subset  \cup_{j}(H^j-H^j_0)$, the restriction of $\iota$ to the subgroup of $G$ generated by $G_0$ and $R$ is injective. Since the union of all such subgroups of $G$ is $G$, so $\iota$ is injective. 
\end{proof}

\section{\large Maximal special abelian splitting}\label{msas}

In this section, we assume \ref{standing assumption} and \ref{SA2}. We further assume $G$ is finitely generated.

As explained in Section \ref{s:sandp}, we want to run the shortening argument in the non-abelian vertex groups of an elliptic (parabolic) splitting $\mathbb{G}$ of $G$. However Theorem \ref{GMWtree} might apply again to these vertex groups, in which case we need to refine $\mathbb{G}$ using the elliptic (or parabolic) splittings produced by Theorem \ref{GMWtree} before we can try to run the shortening argument again.  We introduce a class of splittings (Definition \ref{ssas}) which includes all the refinements of elliptic (parabolic) splittings mentioned above and is closed under further refinement induced by Theorem \ref{GMWtree}. The goal of this section is to show that refining splittings in this class using Theorem \ref{GMWtree} increases the ``complexity'' of the splittings (Proposition \ref{coreproposition}). We measure the complexity of splittings in two different ways: the number of vertices and AFF-rank (Definition \ref{d: AFF-rank}) and we show that there is an upper bound on both of these complexities for the splittings that we consider and hence the refining process eventually stops.

\begin{definition}
Let $H$ be an elliptic (parabolic) subgroup of $G$ and $h, h'\in H$. We say $h$ is {\em smaller than or equal to }$h'$ (or $h'$ is {\em greater  than or equal to }$h$ ) with respect to $\{f_i\}$ and write  $h\leq h'$ (or $h'\geq h$) if the translation length of $f_i(h)$ on any (and hence all ) circle (horocircle) centered at the fixed point of $f_i(H)$ is smaller than or equal to that of $f_i(h')$  \walmost surely. 
\end{definition}

Note that the relation $\leq$ is preserved when passing to subsequences of $\{f_i\}$ with full $\omega$-measure, i.e. if $\{f_{l_i}\}$ is a subsequence with full $\omega$-measure, then $h\leq h'$ with respect to $\{f_i\}$ if and only if $h\leq h'$ with respect to $\{f_{l_i}\}$. 

\begin{definition}
Let $H$ be an elliptic (or parabolic) subgroup of $G$. We say that a subgroup $A$ of $H$ is a {\em $\leq$-closed subgroup of $H$}  if $h'\in A$, $h\in H$ and $h\leq h'$ imply $h\in A$. 
\end{definition}

\begin{lemma}\label{monotone}
Let $H$ be an elliptic (or parabolic) subgroup of $G$ and $A_1$ and $A_2$ be $\leq$-closed subgroup of $H$. Then either $A_1\subset A_2$ or $A_2\subset A_1$. 
\end{lemma}

\begin{proof}
Suppose neither $A_1$ nor $A_2$ contains the other. Then there exists $h_1\in A_1$ and $h_2\in A_2$ such that $h_1\notin A_2$ and $h_2\notin A_1$. Since $\leq$ defines a linear order on all the elements of $H$, we have either $h_1\leq h_2$ or $h_2\leq h_1$. Without loss of generality we assume $h_1\leq h_2$. Since $A_2$ is $\leq$-closed, we have $h_1\in A_2$. Contradiction. 
\end{proof}

\begin{lemma}\label{torsionfreeclosed}
Suppose $A$ is a small elliptic (or parabolic) subgroup of $G$ and $E$ is a $\leq$-closed subgroup of $A$. Then $A/E$ is torsion free.  
\end{lemma}

\begin{proof}
Suppose $A/E$ has torsion. Then there exists $h\in A-E$ such that $h^d\in E$ for some $d\leq 2$. Since $A$ is small elliptic, the \w limit of the angle of rotation of $f_i(h)$ is zero.  Hence \walmost surely the translation length of $f_i(h^d)=f_i(h)^d$ on the circle  centered at the fixed point of $f_i(A)$ is larger than that of $f_i(h)$. Therefore $h\leq h^d$. Since $E$ is $\leq$-closed and $h^d\in E$, we have $h\in E$. Contradiction. The proof of the case where $A$ is parabolic is similar. 
\end{proof}

\begin{definition}\label{ssas}
A splitting $\mathbb{G}$ of $G$ is called a {\em (small) special abelian splitting} of $G$ with respect to $\{f_i: G\rightarrow \Gamma_i\}$ if the following are true:  
\begin{enumerate}
\item $\mathbb{G}$ is bipartite graph with one type of vertex groups being abelian and the other non-abelian. 
\item The abelian vertex groups are (small) elliptic or parabolic with respect to $f_i$. 
\item Edge groups adjacent to an elliptic abelian vertex group are small elliptic and each edge group is a $\leq$-closed subgroup of the abelian vertex group containing it.
\item Each edge group is maximal abelian in the adjacent non-abelian vertex groups . 
\item Any two edge groups in a non-abelian vertex group $N$ are not conjugate to each other in $N$.
\end{enumerate} 
\end{definition}

\begin{definition}\label{ssasrefinement}
Given a special abelian splitting $\mathbb{G}$ of $G$ and an elliptic or parabolic splitting $\mathbb{N}$ of a non-abelian vertex group $N$ of $\mathbb{G}$ relative to its adjacent edge groups. Use $\mathbb{N}$ to refine $\mathbb{G}$ to obtain a splitting $\mathbb{G}_I$. Collapse any edge connecting two abelian vertex groups of $\mathbb{G}_I$. We call the resulting splitting of $G$ the {\em EP-refinement of $\mathbb{G}$ induced by $\mathbb{N}$} and denote it by $\mathbb{G}(\mathbb{N})$. If $\mathbb{N}$ is an amalgamated product and its center abelian vertex group is conjugate to an adjacent edge group of $N$, then we call $\mathbb{G}(\mathbb{N})$ an {\em one-edge EP-refinement of $\mathbb{G}$}.
\end{definition}

\begin{remark}\label{r:collapsing}
Note that the center vertex group $H$ of $\mathbb{N}$ is maximal abelian in $N$ and all adjacent edge groups of $N$ in $\mathbb{G}$ are also maximal abelian in $N$. If $H$ is conjugate to an adjacent edge group $A$ of $N$ in $\mathbb{G}$, then the edge $e$ corresponding to $A$ is connected to the vertex corresponding to $H$ and $e$ is collapsed when we obtain $\mathbb{G}(\mathbb{N})$ from $\mathbb{G}_I$. Note that there is at most one such edge by (5) of Definition \ref{ssas}. Since both $A$ and $H$ are maximal abelian, $A=H$. Hence the vertex group resulting from collapsing $e$ equals the other abelian vertex group adjacent to $e$. 
\end{remark}

\begin{lemma}
An elliptic (or parabolic) splitting is a special abelian splitting. The class of special abelian splittings is closed under refinement induced by elliptic or parabolic splittings of non-abelian vertex groups relative to their adjacent edge groups. 
\end{lemma}

\begin{proof}
Let $\mathbb{G}$ be a elliptic (or parabolic) splitting of $G$. It follows directly from Theorem \ref{GMWtree} that $\mathbb{G}$ satisfies (1), (2), (4) and (5) of Definition \ref{ssas}. If the center vertex group $H$ of $\mathbb{G}$ is parabolic, then $\mathbb{G}$ clearly satisfies (3).  Suppose $H$ is elliptic. Let $H_0$ be the edge group. By Lemma \ref{smallrotation}, $H_0$ is small elliptic. We claim that $H_0$ is a $\leq$-closed subgroup of $H$. (Proof of the claim: Let $Y$ be the fixed ball of $H$. Then $h\in H_0$ if and only if $h$ fixes $Y$ point-wise. Suppose $h\in H_0$ and $h'\leq h$. Then since the rotation angle of $f_i(h')$ is smaller than that of $f_i(h)$ \walmost surely, we know that $h'$ fixes $Y$ and hence $h'\in H_0$. ) As a result, $\mathbb{G}$ also satisfies (3). Therefore $\mathbb{G}$ is a special abelian splitting. 

Now let $\mathbb{G}$ be a special abelian splitting and $\mathbb{N}$ be an elliptic or parabolic splitting of one of $\mathbb{G}$'s non-abelian vertex groups $N$. Consider $\mathbb{G}(\mathbb{N})$.  From Remark \ref{r:collapsing} it is easy to see that $\mathbb{G}(\mathbb{N})$ satisfies (1)and (2) of Definition \ref{ssas}. Note that $\mathbb{G}$ and $\mathbb{N}$ satisfy (3). Hence each edge group of $\mathbb{G}(\mathbb{N})$ satisfies (3) as it equals an edge group of either $\mathbb{G}$ or $\mathbb{N}$. Let $N'$ be a non-abelian vertex group of $\mathbb{N}$. An edge groups of $N'$ is either an edge group of $\mathbb{G}$, which is maximal abelian in $N$ and hence in $N'$, or an edge group of $\mathbb{N}$, which is also maximal abelian in $N'$ by Theorem \ref{GMWtree}. So $\mathbb{G}(\mathbb{N})$ satisfies (4). Let $E_1$ and $E_2$ be adjacent edge groups of $N'$. If $E_1$ and $E_2$ are both edge groups of $\mathbb{G}$, then they are not conjugate to each other in $N$ and hence $N'$ as $\mathbb{G}$ satisfies (5). If $E_1$ and $E_2$ are both edge groups of $\mathbb{N}$, then they satisfy (5) by Theorem \ref{GMWtree}. Suppose $E_1$ is an edge group of $\mathbb{G}$ and $E_2$ is an edge group of $\mathbb{N}$. If  they are conjugate to each other in $N'$, then by Remark \ref{r:collapsing} they are the same edge group. Contradiction. Therefore $\mathbb{G}(\mathbb{N})$ also satisfies (5).  
\end{proof}

Our next goal is to give an upper bound on the number of edges of a Special abelian splitting of $G$. 

\begin{definition}
Let $\mathbb{G}$ be a splitting of a group $G$ and $T$ be its Bass-Serre tree. We say that $\mathbb{G}$ and $T$ are {\em $k$-acylindrical} if 
\begin{enumerate}
\item $\mathbb{G}$ is reduced.
\item $T$ is minimal. 
\item For any non-trivial element $g\in G$, the set of points in $T$ fixed by $g$ has diameter bounded by $k$.
\end{enumerate}
\end{definition}

Recall that a splitting (graph of groups decomposition) $\mathbb{G}$ of a group $G$ is {\em reduced} if the vertex group of every vertex of valence at most 2 properly contains the edge group(s) of the edge(s) incident to it.

\begin{theorem}\label{AA}(Acylindrical Accessibility: \cite[Theorem 4.1]{Sela1}, Weidmann \cite{Weidmann}).
Let $G$ be a non-cyclic freely indecomposable finitely generated group and let $T$ be a minimal $k$-acylindrical simplicial $G$-tree. Then $T/G$ has at most $K$ edges, where $K$ depends only on $k$ and $G$. 
\end{theorem}

\begin{proposition}\label{2aa}
Special abelian splittings of $G$ are 2-acylindrical. 
\end{proposition}

\begin{proof}
Let $\mathbb{G}$ be an special abelian splitting of $G$ and let $T$ be its Bass-Serre tree. Suppose $1\neq g\in G$ fixes three consecutive edges $e_1$, $e_2$ and $e_3$ of $T$ and let $v$ and $w$ be $e_1\cap e_2$ and $e_2\cap e_3$, respectively. Without loss of generality, we can assume $v$ have non-abelian stabilizer $G_v$ since vertices of the same type are not adjacent to each other. Let $H_1$ and $H_2$ be the stabilizers of $e_1$ and $e_2$, respectively.  Then $g\in H_1\cap H_2$. Note that $H_1$ and $H_2$ are both maximal abelian in $G_v$. Since $g\neq 1$ and $G_v\subset G$ is commutative transitive, we know that $H_1=H_2$. By (8) of Definition \ref{ssas} this can only happen when $H_1$ and $H_2$ correspond to the same edge in $\mathbb{G}$. Hence $H_1$ is a conjugate of $H_2$ by some $t\in G_v-H_1$. Let $h\in H_1$. Then \walmost surely $f_i(h)$ is in some elliptic (parabolic) subgroup $E_i$ of $\Gamma_i$. Let $h'=tht^{-1}$. Then $h'\in H_2=H_1$. So $f_i(h')\in E_i$ \walmost surely. Hence \walmost surely $f_i(t)$ conjugate an element in $E_i$ to an element in $E_i$. This implies that $f_i(t)\in E_i$. Therefore $t\in H_1$, which contradicts our choice of $t$.  Therefore $g=1$. 
\end{proof}

\begin{corollary}\label{boundonnumofedges}
Suppose $G$ is freely indecomposable. Then the number of edges of a special abelian splitting of $G$ is no more than $K$, where $K$ depends only on $G$. 
\end{corollary}

\begin{proof}
Let $\mathbb{G}$ be a special abelian splitting of $G$. For each valence two abelian vertex group $A$, collapse one of its adjacent edge if the corresponding edge group equals $A$. Note that if both of $A$'s adjacent edge groups equal $A$, we still only collapse one of them. Let $\mathbb{G}'$ be the resulting splitting. Note that $\mathbb{G}'$ is reduced and its number of edges is at least half of that of $\mathbb{G}$. Since $\mathbb{G}$ is 2-acylindrical by Lemma \ref{2aa}, $\mathbb{G}'$ is also two 2-acylindrical. Hence by Theorem \ref{AA} there is an upper bound $K'$ on the number of edges of $\mathbb{G}'$ and $K'$ depends only on $G$. Hence $K=2K'$ is an upper bound on the number of edges of $\mathbb{G}$. 
\end{proof}

Our next goal is to define the second type of complexity of special abelian splittings and show that there is an upper bound on it.

\begin{definition}\label{d: AFF-rank}
Let $\mathbb{G}$ be a special abelian splitting of $G$. Let $A$ be an abelian vertex group of $\mathbb{G}$ and let $E$ be the subgroup of $A$ generated by all adjacent edge groups. We call the rank of $A/E$ the {\em abelian rank of $\mathbb{G}$ at $A$}.  We define the {\em abelian rank} of $\mathbb{G}$, denoted by ${\rm Rank}_A(\mathbb{G})$, to be the sum of the abelian rank of $\mathbb{G}$ at each of its abelian vertex group. 
\end{definition}

\begin{remark}
By Lemma \ref{monotone} and Definition \ref{ssas}, the adjacent edge groups of $A$ form a monotone sequence in terms of containment. Hence the subgroup of $A$ generated by all adjacent edge groups equals the maximal adjacent edge group of $A$. 
\end{remark}

\begin{lemma}\label{boundedAFF}
Let  $\mathbb{G}$ be a special abelian splitting. Then ${\rm Rank}_A(\mathbb{G})\leq D$, where $D$ is the rank of $G$. 
\end{lemma}

\begin{proof}
Let $D$ be the rank of $G$. Collapsing all non-abelian vertex groups of $\mathbb{G}$ induces a surjective homomorphism from $G^s$ to $\oplus A/E$, where the direct sum is over all abelian vertex group $A$ of $\mathbb{G}$ and its maximal adjacent edge group $E$. It follows that ${\rm Rank}_{A}(\mathbb{G})\leq D$. 
\end{proof}

\begin{definition}\label{goodrefine}
Given a special abelian splitting $\mathbb{G}$ of $G$ and a non-abelian vertex group $N$ of $\mathbb{G}$. Let $S_{N}$ be a finite subset of $N$. Suppose $S_{N}$ intersects with each of the adjacent edge groups of $N$ in $\mathbb{G}$ and it generates $N$ together with them.  Use Proposition \ref{relativewelldefined} to obtain a splitting of $N$ relative to its adjacent edge groups. Use $\mathbb{N}$ to refine $\mathbb{G}$ to obtain a splitting $\mathbb{G}_I$. Collapse any edge connecting two abelian vertex groups of $\mathbb{G}_I$. We call the resulting splitting of $G$ a {\em good-refinement of $\mathbb{G}$ with respect to $N$} and denote it by $\mathbb{G}(N)$. 
\end{definition}

\begin{remark}
Although $\mathbb{G}(N)$ depends on the choice of $S_N$, the difference caused by different choices of $S_N$ is not important to our application of good refinement. Also, such $S_N$ always exists. This are the reasons why we do not mention $S_N$ is the terminology and notation of good-refinement. 
\end{remark}

The point of Good-refinement is that $S_N$ is now a good relative generating set for the vertex group of $\mathbb{G}(N)$ containing it.  Note that a good refinement preserves the underlying graph and all abelian vertex groups and it replaces a non-abelian vertex group and its adjacent edge groups by their subgroups. This observation yields the following:  

\begin{lemma}
With the notation as in Definition \ref{goodrefine}, we have ${\rm Rank}_{A}(\mathbb{G})\leq {\rm Rank}_{A}(\mathbb{G}(N))$. 
\end{lemma}

The last task of this section is the following technical result.

\begin{proposition}\label{coreproposition}
Assume \ref{standing assumption} and \ref{SA2} and $G$ is freely indecomposable. Let $\mathbb{G}$ be a special abelian splitting of $G$ with respect to $\{f_i\}$. Suppose 
\begin{enumerate}
\item For each non-abelian vertex group $N$,  a good relative generating set $S_N$ relative to the adjacent subgroups of $N$ exists.  
\item $|f_i|_{S_N}\rightarrow\infty$ for at least one non-abelian vertex group $N$. 
\item Theorem \ref{GMWtree} applies to any non-abelian vertex groups $N$ with $|f_i|_{S_N}\rightarrow\infty$ relative to its adjacent edge groups.
\end{enumerate}
Then there exists a special abelian splitting $\mathbb{G}'$ of $G$, which either has more edges or has bigger abelian rank than $\mathbb{G}$.  
\end{proposition}

\begin{proof}

If a non-abelian vertex group $N$ has an elliptic or parabolic splitting $\mathbb{N}$ with at least two edges, then $\mathbb{G}(\mathbb{N})$, the EP-refinement of $\mathbb{G}$ induced by $\mathbb{N}$, has at least one more edge than $\mathbb{G}$. If $\mathbb{N}$ is a one-edge splitting with non-abelian vertex group $N'$ and all adjacent edge groups of $N$ are conjugate to subgroups of $N'$, then $\mathbb{G}(\mathbb{N})$ also has one more edge than $\mathbb{G}$. Hence the statement of the Proposition is true in the above cases. 

For the rest of the proof, we assume that any elliptic or parabolic splitting of any non-abelian vertex group $N$ is a one-edge splitting whose abelian vertex group is conjugate to an adjacent edge group of $N$. Under this assumption any EP-refinement of $\mathbb{G}$ has the same underlying graph as $\mathbb{G}$. The effect of an EP-refinement is the following: non-abelian vertex groups and edge groups are replaced by their subgroups; the boundary monomorphisms are restricted to the corresponding subgroup and modified by a conjugation. 

We think of all vertex groups and edge groups of $\mathbb{G}$ as subgroups of $G$. 
Fix a maximal subtree of $T$ of the underlying graph of $\mathbb{G}$. For any edge in $T$, we identify the corresponding edge group with its images under the boundary monomorphisms in both of its adjacent vertex groups. For any edge not in $T$, we identify the corresponding edge group with its image in the adjacent abelian vertex group.  

Let $N$ be a non-abelian vertex group of $\mathbb{G}$ such that \walmost surely $| f_i |_{S_N}\geq | f_i |_{S_{N'}}$ for any non-abelian vertex group $N'$ of $\mathbb{G}$.  By the assumption of the proposition, Theorem \ref{GMWtree} can be applied to $N$ relative to its adjacent edge groups. Let $\mathbb{N}$ be the resulting one-edge elliptic or parabolic splitting, whose center vertex group, non-abelian vertex group and edge group are denoted by $A$, $N_0$ and $E$, respectively. By assumption one of the adjacent edge groups of $N$ is $A$ up to conjugation. Going from $\mathbb{G}$ to $\mathbb{G}(\mathbb{N})$, $N$ is replaced by $N_0$ and $A$ is replaced by $E$. Let $B$ be the abelian vertex group of $\mathbb{G}(\mathbb{N})$ adjacent to the edge group $E$. Let $\mathbb{G}_1$ be a Good refinement of $\mathbb{G}(\mathbb{N})$ with respect to $N_0$. Let $E', A_1, \dots, A_m$ be edge groups in $\mathbb{G}_1$ adjacent to $B$. Here $E'$ is a subgroup of $E$ and corresponds to the same edge as $E$.   Our goal is to show that performing a sequence of EP-refinements and Good-refinements on $\mathbb{G}_1$ can produce a splitting $\mathbb{G}'$ satisfying:
\begin{enumerate}
\item All abelian vertex groups are not changed.
\item The adjacent edge groups of all abelian vertex groups are either not changed or replaced by a proper subgroup. 
\item Each $A_j$ is replaced by a subgroup of $E$.     
\end{enumerate}

 Note that (1) and (2) are obvious consequences of EP-refinements and Good-refinements. 
 
 The key in proving (3) is the following: Since $\mathbb{N}$ is a non-trivial one-edge splitting, $A$ properly contains $E$. since $G$ is freely indecomposable, $E\neq\{1\}$.  Let $\mathcal{X}$ be the asymptotic cone of $\mathbb{H}$ defined by $\{| f_i |_{S_N}\}$, $\{x_i\}$ and $\omega$, where $x_i$ is a centrally located point with respect to $f_i$ and $S_N$. Let $h\in A-E$ and $1\neq h'\in E$. By Theorem \ref{GMWtree}, $h$ does not fixes points on the boundary of ${\rm Fix}_{\mc{X}}(h')$. Note that by Lemma \ref{ball} or Lemma \ref{horoball}, points on the boundary of ${\rm Fix}_{\mc{X}}(h')$ has the form $y=[\{y_i\}]$, where $y_i\in\partial B_i(h')$. Therefore for any $y_i\in B_i(h')$, we have 
 \begin{equation}\label{e:move y}
 {\rm lim}^{\omega}d_{\mathbb{H}}(y_i, f_i(h)y_i)/| f_i |_{S_N}>0. 
\end{equation}

  We now using EP-refinement to ``shrink'' $A_1$ to a subgroup of $E$:  Since both $A_1$ and $E$ are $\leq$-closed, by Lemma \ref{monotone}, we have either $A_1\subset E$ or $E\subset A_1$. Suppose $A_1$ properly contains $E$.  (Otherwise we are done.) Let $N_1$ be the non-abelian vertex group of $\mathbb{G}_1$ adjacent to $A_1$. Note that $N_1$ is a subgroup of a non-abelian vertex group $N'$ adjacent to $B$ in $\mathbb{G}$. Let $S_1$ be  a finite good relative generating set of $N_1$. Since $S_{N'}$ is a good generating set of $N'$, for each $g\in S_1$, we have $d_{\mathbb{H}}(x_i, gx_i)\leq D_g|f_i|_{S_{N'}}$, where $D_g$ is independent of $i$. Hence we have ${\rm max}_{g\in S_1}d_{\mathbb{H}}(x_i, gx_i)\leq D|f_i|_{S_{N'}}$ for some $D$ independent of $i$. Therefore we have $|f_i|_{S_1}\leq D|f_i|_{S_{N'}}$. By our choice of $N$, we have $|f_i|_{S_{N'}}\leq |f_i|_{S_N}$. As a result, we have
 \begin{equation}\label{e:tran len compar}
 |f_i|_{S_1}\leq D|f_i|_{S_N}
 \end{equation}
 
 Let $b$ be the boundary monomorphism from $A_1$ to $N_1$.  Then $b(A_1)=tA_1t^{-1}$. Note that $t=1$ if and only if $A_1$ corresponds to an edge in $T$.  Since $A_1$ properly contains $E$ and $E\neq \{1\}$, we have $h_1\in b(A_1)-b(E)$ and $1\neq h'_1\in b(E)$. Then $h_1=tht^{-1}$ and $h'_1=th't^{-1}$ for some $h\in A_1-E$ and $1\neq h'\in E$. If $h\in A$, (\ref{e:move y}) holds. Suppose $h\in A_1-A$. Since $A$ is $\leq$-closed, $h$ moves points in $\partial B_i(h')$ further than any element in $A$ does. Hence again (\ref{e:move y}) holds. Note that $h_1$ moves points on $\partial B_i(h'_1)$ the same amount as $h$ moves points on $\partial B_i(h')$. Hence by (\ref{e:move y}) and (\ref{e:tran len compar}) for any $y'_i\in \partial B_i(h'_1)$ we have 
 \begin{equation}\label{e:same ttttt}
 {\rm lim}^{\omega}d_{\mathbb{H}}(y'_i, f_i(h_1)y'_i)/|f_i|_{S_1}\geq {\rm lim}^{\omega}d_{\mathbb{H}}(y_i, f_i(h)y_i)/D|f_i |_{S_N}>0. 
 \end{equation}
 
Consider the case where ${\rm lim}^\omega |f_i|_{S_1}=\infty$. Let $\mc{X}_1$ be the asymptotic cone of $\mathbb{H}$ defined by $\{| f_i |_{S_1}\}$, $\{x^1_i\}$ and $\omega$, where $x^1_i$ is a centrally located point with respect to $f_i$ and $S_1$. (\ref{e:same ttttt}) implies that there exists $y\in {\rm Fix}_{\mathcal{X}_1}(h'_1)$ such that $y\notin{\rm Fix}_{\mathcal{X}_1}(h_1)$. Hence by Lemma \ref{nondearcfixed}, Theorem \ref{GMWtree} can be applied to $N_1$ to produce an elliptic or parabolic splitting $\mathbb{N}_1$ with center vertex group being $b(A_1)$. By assumption of our proof, $\mathbb{N}_1$ is a one edge splitting. Since ${\rm Fix}_{\mathcal{X}_1}(h_1)$ is a proper subset of ${\rm Fix}_{\mathcal{X}_1}(h'_1)$ for any $h_1\in b(A_1)-b(E)$, the edge group of $\mathbb{N}_1$ is contained in $b(E)$. Hence the one-edge EP refinement of $\mathbb{G}_1(\mathbb{N}_1)$ replaces $A_1$ by a subgroup of $E$. We then perform a Good refinement to $\mathbb{G}_1(\mathbb{N}_1)$ to obtain $\mathbb{G}_2$.

Now suppose $|f_i|_{S_1}$ has finite $\omega$-limit.  Let $S$ be a finite subset of $N_1$ containing $h_1, h_1'$ but not contained in $A_1$. Let $M=<S>$. Then $M$ is not abelian since $S$ is not contained in $A_1$. Since $S_1$ is a good relative generating set, the fact that $|f_i|_{S_1}$ has finite $\omega$-limit implies that $|f_i|_{S}$ has finite $\omega$-limit. Apply Theorem \ref{strongversionofboundedcase} to $\{f_i|_M\}$, we have $f_i=\pi^i\circ \tilde c_i\circ f'$, where $\tilde c_i$ is a conjugation  of $\Gamma$ and $f': M\rightarrow \Gamma$. Since $f_i(h_1)$ is elliptic or parabolic, $\tilde c_i\circ f'(h_1)$ is parabolic. Let $\tilde\gamma_i$ be the generator of the maximal parabolic subgroup of $\Gamma$ containing $\tilde c_i\circ f'(h_1)$. Then $\tilde c_i\circ f'(h_1)=(\tilde\gamma_i)^k$ for some $k$. Note that $k$ does not depend on $i$. So $f_i(h_1)=\gamma_i^k$, where $\gamma_i=\pi^i(\tilde\gamma_i)$. Hence $f_i(h)=(f_i(t^{-1})\gamma_if_i(t))^k$. Similarly, $f_i(h')=(f_i(t^{-1})\gamma_if_i(t))^{k'}$ for an integer $k'$ independent of $i$. Hence by Lemma \ref{fixedpowerrotation}, ${\rm Fix}_{\mathcal{X}}(h)={\rm Fix}_{\mathcal{X}}(h')$. This contradicts the choice of $h, h'$. 

Repeat the same process for $A_2, \dots, A_m$. Then the resulting splitting $\mathbb{G}'$ has the desired properties. 
 
We now show that $\mathbb{G}'$ has bigger abelian rank than $\mathbb{G}$.
Property (1) and (2) of $\mathbb{G}'$ implies that the abelian rank of $\mathbb{G}'$ at any abelian vertex group is at least as big as that of $\mathbb{G}'$. Hene it suffices to show that the abelian rank of $\mathbb{G}'$ at $B$ is strictly bigger than that of $\mathbb{G}$. By Property (3) of $\mathbb{G}'$, the abelian rank of $\mathbb{G}'$ at $B$ is at least ${\rm rank}(B/E)$. The abelian rank of $\mathbb{G}$ at $B$ is ${\rm rank}(B/A)$. Note that we have $1\rightarrow B/A\rightarrow B/E\rightarrow A/E\rightarrow 1$. Since $E$ is a proper subgroup of $A$ and $A$ is small elliptic, so by Lemma \ref{torsionfreeclosed}, $A/E$ is a non-trivial torsion free abelian group. Hence the rank of $A/E$ is positive. Therefore the abelian rank of $\mathbb{G}'$ at $B$ is strictly bigger than that of $\mathbb{G}$.
\end{proof}

\vspace{3mm}

\section{\large The shortening argument}\label{hs}

Let $G$ be a countable group. Suppose $S$ is a finite subset of $G$ and $K$ is a subgroup of the automorphism group of $G$. Let $\mathcal{H}$ be a collection of subgroups of $G$. Denote by ${\rm Aut}(G, \mathcal{H})$ the collection of automorphisms of $G$ that act on each $H\in \mc{H}$ by conjugation. 

\begin{definition}\label{d:short}
Let $\Gamma$ be a Fuchsian group. A homomorphism $f:G\rightarrow \Gamma$ is called {\em short} with respect to $K$ and $S$ if for any $\alpha\in K$, we have $|f|_S\leq |f\circ \alpha|_S$.
\end{definition}

\begin{lemma}
Suppose $f(S)$ is not contained in a parabolic subgroup of $\Gamma$. Then there exists $\alpha\in K$ such that $f\circ\alpha$ is short with respect to $K$ and $S$.  
\end{lemma}

\begin{proof}
Let $l={\rm inf}\{|f\circ \alpha|_S\mid \alpha \in K\}$. Let $\alpha_n$ be a sequence of automorphisms in $K$ such that $|f\circ\alpha_n|_S$ converges to $l$. Then we have $|f\circ\alpha_n|_S\leq D$ for some $D$ independent of $n$. Let $x_n$ be a centrally located point for $f\circ\alpha_n$. Let $\Omega$ be a fundamental domain of $\Gamma$. Post compose $f\circ\alpha_n$ with a conjugation $c^n$ of $\Gamma$ if necessary, we can assume $x_n\in \Gamma$ for all $n$. Note that post composing $f\circ\alpha_n$ with a conjugation does not change its translation length. So we have a sequence of homomorphisms $f_n=c_n^f\circ\alpha_n$ with a central located point in $\Omega$ and $|f_n|_S$  is bounded. Suppose $x_n$ is unbounded. Since $\{x_n\}$ is in $\Omega$, a subsequence of it converge to a point $x_{\infty}$, which is the fixed point of a parabolic element in $\Gamma$. This implies that for some $n$ large enough $f_n(g)$ is a parabolic element with fixed point $x_{\infty}$ for any $g\in S$ since $d_{\mathbb{H}}(f_n(g)x_n, x_n)\leq |f'_n|_S\leq D$. Hence $f_n(S)$ is contained in the parabolic subgroup of $\Gamma$ fixing $x_\infty$. Contradiction. Therefore $\{x_n\}$ is bounded. As a result, there exists a point $x\in \Omega$ such that for any $g\in S$,  $f_n(g)$ moves $x$ by an amount bounded independent of $n$. So $\{f_n(g)\}$ is contained in a compact subset of ${\rm PSL}(2, \mathbb{R})$. So a subsequence of $f_n(g)$ converges. Since $\{f_n(g)\}\subset \Gamma$ is a sequence in a discrete group, so $f_n(g)$ is the same element in $\Gamma$ for all big enough $n$. This implies that for $n$ large enough $|f_n|_S$ equals $l$. Note that $|f\circ\alpha_n|_S=|f_n|_S$. So $f\circ\alpha_n$ is short. 
\end{proof}

\begin{proposition}\label{shortening}
Assume \ref{standing assumption}. Let $S$ be a good generating set of $G$ relative to a finite collection $\mathcal{H}$ of elliptic or parabolic subgroups. $\mathcal{H}$. Suppose that 
\begin{enumerate}
\item $G$ is freely indecomposable relative to $\mathcal{H}$;  
\item $f_i$ is short  with respect to ${\rm Aut}(G, \mathcal{H})$;
\item ${\rm lim}^\omega|f_i|_S=\infty$; 
\item None of the non-trivial segment stabilizer of $\mc{T}$, the minimal $G$-invariant subtree of  the asymptotic cone $\mathcal{X}$ of $\mathbb{H}$ defined by $\{x_i\}$, $\{|f_i|_S\}$ and $\omega$, is elliptic or parabolic. 
\end{enumerate}
Then $Ker^\omega(f_i)\neq \{1\}$.
\end{proposition}

\begin{proof}
Suppose $Ker^\omega(f_i)= \{1\}$. Since $f_i$ is assumed to be short, to get a contradiction, it suffices to find automorphisms $\alpha_i\in{\rm Aut}(G, \mathcal{H})$ such that $|f_i\circ \alpha_i|_S<|f_i|_S$.

By assumption, none of the non-trivial segment stabilizer is elliptic or parabolic, we know that stabilizers of tripods are trivial by Lemma \ref{tripod sta elliptic}. This together with the fact that segment stabilizer are abelian (Lemma \ref{arc sta abel}) imply that stabilizers of unstable arcs are trivial.  Note that each $H_j$ fixes a point in $\mc{T}$ by Lemma \ref{Ell and para has fixed point}.  Hence the Rips machine \cite[Theorem 5.1]{Guirardel} can be applied to the action of $G$ on $\mathcal{T}$. Since the stabilizers of unstable arcs and the stabilizers of (infinite) tripods are trivial and ($G$, $\mathcal{H}$) is freely indecomposable, the action of $G$ on $\mathcal{T}$ admits a graph of actions  decomposition $\mathcal{G}(\mathbb{A})$ with each vertex action being one of the three types: simplicial, orbifold (IET) or axial.  

Suppose $[x, gx]$ has non-degenerate intersection with an orbifold component $T_{\tilde v}$. Note an orbifold component contains a non-trivial tripod. Hence the kernel of the action of $G_v$ on $T_{\tilde v}$ is faithful. Therefore \cite[Theorem 5.1]{RS1} gives what we need. 

Now suppose $[x, gx]$ has non-degenerate intersection with an axial component $T_{\tilde v}$. Let $G_v$ be the vertex group of $\mathbb{A}$ corresponding to the orbit of $T_{\tilde v}$. Without loss of generality, we assume $G_v$ preserves $T_{\tilde v}$, which is isometric to a line. Let $G^+_v$ be the index-two subgroup of $G_v$ which preserves the ends of $T_{\tilde v}$ and $K$ be the kernel of the action of $G^+_v$ on  $T_{\tilde v}$. By Lemma \ref{arc sta abel}, $K$ is abelian. By the assumption of the proposition, $K$ is neither elliptic nor parabolic. So for any $g\in K$ $f_i(g)$ is hyperbolic \walmost surely. Since $K$ is normal in $G^+_v$,  for any $h\in G^+_v$, we have $hgh^{-1}\in K$. Hence $f_i(hgh^{-1})$ and $f_i(g)$ has the same axis in $\mathbb{H}$. Therefore $f_i(h)$ preserves the same axis.  Note that since $h$ preserves the orientation of $T_{\tilde v}$, $f_i(h)$ is not rotation of order two. Hence $f_i(h)$ is also hyperbolic with the same axis. So $f_i(g)$ and $f_i(h)$ commute \walmost surely. As a result, $g$ and $h$ commute and $G^+_v$ is abelian. Having the above information about $G_v$, one can check that the rest of the proof in the axial case is covered by the argument in \cite[Section 4.2.1]{R}.  

Finally we consider the case where $[x, gx]$ has non-degenerate intersection with a simplicial component $T_{\tilde v}$.  Note that for any edge $e$ in $T_{\tilde v}$, the stabilizer $A_e$ is abelian by Lemma \ref{arc sta abel}. By the assumption of the proposition $A_e$ is not elliptic or parabolic. So $A_e$ is hyperbolic. In particular, it has infinite order.  Hence the center of $A_e$, which equals $A_e$ itself, contains an element of infinite order.  Now we note that \cite[Proposition 4.18]{R} and the rest of the shortening argument follows exactly the same as in section 4.2.3 in \cite{R}.
\end{proof}

\section{\large Proof of Theorem \ref{MT2}}\label{proof}

In this section, we proof Theorem \ref{MT2}. We recall the definition of Dehn twists.

\begin{definition}
Let $\mathbb{G}$ be a splitting of $G$. Suppose $e$ is a separating edge of $\mathbb{G}$, i.e. $\mathbb{G}-e$ has two connected components $\mathbb{G}_1$ and $\mathbb{G}_2$. Let $C$ be the edge group of $e$.  Let $z$ be an element of the centralizer $Z_{G}(C)$ of $C$ in $G$. The automorphism $\alpha_z$ of $G$, called the Dehn twist along $e$ by $z$, is determined as follows. 
\begin{equation*}  \alpha_z(g)=
\left\{ \begin{array}{ll} g, \hspace{12mm}   g\in \pi_1(\mathbb{G}_1)
                                \\ zgz^{-1}, \hspace{4mm}   g\in \pi_1(\mathbb{G}_2)
                                
                         \end{array} \right.    
                         \end{equation*} 
\end{definition}

\begin{remark}
Dehn twists are also defined for non-separating edges. (See \cite[Definition 3.14]{R} for the general definition.) Essentially we are also using them in the next proof. However we think of them as extensions of the Dehn twists defined above.  
\end{remark}

\begin{proposition}\label{vertexbounded}
Assume \ref{standing assumption}. Let $\mathbb{G}$ be a special abelian splitting of $G$ with respect to $\{f_i\}$ such that each non-abelian vertex group $N$ of $\mathbb{G}$ has a good generating set $S_N$ relative to its adjacent edge groups. Suppose $ker^{\omega}(f_i)=\{1\}$ and ${\rm lim}^\omega|f_i|_{S_N}<\infty$ is bounded for each non-abelian vertex group $N$. Then there exists $\bar A=(\mathbb{Z}\oplus(\mathbb{Z}/a_1\mathbb{Z}), \dots, \mathbb{Z}\oplus(\mathbb{Z}/a_k\mathbb{Z}))$ and $\Gamma*_{\bar p}\bar A$ such that $f_i$ factor through $\Gamma*_{\bar p}\bar A$ \walmost surely.  Moreover, the factoring map $\pi_{f_i}: \Gamma*_{\bar p}\bar A\rightarrow \Gamma_i$ is an extension of the natural map $\pi^i$.
\end{proposition}

\begin{proof}

We first show that the edge groups of $\mathbb{G}$ are infinite cyclic. Let $N$ be a non-abelian vertex groups of $\mathbb{G}$ and let $x^N_i\in\mathbb{H}$ be a centrally located point of $f_i$ with respect to $S_N$. Let $E$ be an edge group adjacent to $N$. Denote by $E_i$ the maximal elliptic (or parabolic) subgroup of $\Gamma_i$ containing $f_i(E)$. Let $D_i$ be the Margulis domain  associated to $E_i$ (See Section \ref{s:Bounded}).   Since $|f_i|_{S_N}$ has finite $\omega$-limit and both $S_N\cap E$ and $S_N\cap(N-E)$ are both non-empty, by Lemma \ref{boudnary disct} $d_{\mathbb{H}}(x^N_i, \partial D_i)$ has finite $\omega$-limit. We claim that $E$ is infinite cyclic. Define a map $f: E\rightarrow \mathbb{Z}$ as follow: Let $e_i$ be the generator of $E_i$ with the smallest rotation angle. Then for any $g\in E$, since  both $d_{\mathbb{H}}(x^N_i, f_i(g)x^N_i)$ and $d_{\mathbb{H}}(x^N_i, \partial D_i)$ has finite $\omega$-limit, we know that \walmost surely $f_i(g)=(e_i)^{m_i}$ for some $|m_i|\leq d$, where $d$ depends on $g$ but independent of $i$. Hence \walmost surely, we have $f_i(g)=(e_i)^{m}$ for some $|m|\leq d$. Define $f(g)=m$. One can check that $f$ is a homomorphism as $f_i$ are. Suppose $f(g_1)=f(g_2)$. Then $f_i(g_1)=f_i(g_2)$ \walmost surely. Hence $g_1^{-1}g_2\in ker^\omega(f_i)=\{1\}$. Therefore $g_1=g_2$. So $f$ is injective. Since $G$ is freely indecomposable, $E\neq \{1\}$. Hence $E$ is infinite cyclic.  Since $A$ and $E$ are both arbitrary, as a result of the claim, all the vertex groups of $\mathbb{G}$ are finitely generated. 

Let $G_0\subset G$ be the subgroup generated by all the non-abelian vertex groups of $\mathbb{G}$.  We now show that there exists $\alpha_i\in {\rm Aut}(G)$ preserving $G_0$ such that $|f_i\circ \alpha_i|_S$ has finite $\omega$-limit, where $S$ is a finite generating set of $G_0$.

Let $\mathbb{T}$ be a maximal subtree of (the underlying graph of) $\mathbb{G}$. Let $G_T$ be the subgroup of $G$ corresponding to $\mathbb{T}$, i.e. $G_T$ is the fundamental group of $\mathbb{T}$. Let $N_0$ be a non-abelian vertex group of $\mathbb{G}$. We think $N_0$ as the root of the tree $\mathbb{T}$ and consider each of its branches. Let $\mathbb{T}_1$ be a branch at $N_0$, i.e. the closure of a connected component of $\mathbb{T}-N_0$. For each vertex $v$ (or the corresponding vertex group ) of $\mathbb{T}_1$, we define the {\em level} of $v$ (or the corresponding vertex group ) to be the $\mathbb{T}$ distance between $v$ and $N_0$. Then vertex groups of even levels are non abelian vertex group and vertex groups of odd levels are abelian vertex groups. Let $A$ be the  vertex group of level one and $N'$ be a level two vertex group. Denote the edge between $A$ to $N'$ by $e$ and its edge group by $E$. Let $E_i$ be the maximal elliptic (or parabolic) subgroup of $\Gamma_i$ containing $f_i(E)$.  Let $D_i$ be the Margulis domain associated to $E_i$. As in the last paragraph, we have 
\begin{equation}\label{9e:bound1}
{\rm lim}^\omega d_{\mathbb{H}}(x'_i, \partial D_i)<\infty,
\end{equation} 
where $x'_i$ is a centrally located point of $f_i$ with respect to $S_{N'}$.   
Similarly 
\begin{equation}\label{9e:bound2}
{\rm lim}^\omega d_{\mathbb{H}}(x^0_i, \partial D_i)<\infty, 
\end{equation}
where $x^0_i$ is a centrally located point of $f_i$ with respect to $S_{N_0}$. Let $g\in E$. As in the last paragraph, \walmost surely we have $f_i(g)=(e_i)^{m}$ for some $0<|m|\leq d$, where $e_i$ is the generator of $E_i$ with the smallest rotation angle. By the definition of Margulis domain, $e_i$ moves any point in $\partial D_i$ by the Margulis constant $\epsilon_2$. Hence \walmost surely $f_i(g)$ moves any point in $\partial D_i$ by an amount independent of $i$.  Combine this with (\ref{9e:bound1}) and (\ref{9e:bound2}), we know that for some $l_i\in \mathbb{Z}$, 
\begin{equation}
 {\rm lim}^\omega d_{\mathbb{H}}(f_i(g^{l_i})x'_i, x^0_i)<\infty.
 \end{equation} 
 Let $\alpha_i^{N'}\in {\rm Aut}(G_T)$ be the Dehn twist along $e$ by $g^{l_i}$, which conjugates the subgroup corresponding to the branch of $\mathbb{T}$ containing $N'$ by $g^{l_i}$. Since $\alpha_i^{N'}$ acts on all edge groups in $\mathbb{T}$ by conjugation, it can be extended to an automorphism of $G$, which we still denote by $\alpha_i^{N'}$. Note that $f_i(g^{l_i})x'_i$ is a centrally located point of $f_i\circ\alpha_i^{N'}$ with respect to $S_{N'}$. Hence by precomposing $f_i$ with a Dehn twist, we bring the chosen centrally located point for $N'$ to a bounded distance from $x^0_i$. Note that $\alpha_i^{N'}$ changes only the branch at $A$ containing $N'$ and it also preserves the relative distance between the chosen centrally located points of the non-abelian vertex groups in this branch.  

Similarly, we use Dehn twists to bring all the chosen centrally located points of the non-abelian vertex groups in $\mathbb{T}_1$ of level two to a bounded neighborhood of $x^0_i$. Inductively, we can bring the level-$2n$ centrally located points to a bounded neighborhood of the level-$2(n-1)$. Hence after precomposing $f_i$ with finitely many Dehn twists, there is a centrally located point for each of the non-abelian vertex group in $\mathbb{T}_1$ that is at most $r$ away from $x^0_i$, where $r$ is independent of $i$. Similarly, we achieve the same goal for other branches at $N_0$ by using Dehn twists. Let $\alpha_i$ be the composition of all these Dehn twists. Then for each non-abelian vertex group $N$ there is a centrally located point $x^N_i$ of $f'_i=f_i\circ \alpha_i$ with respect to $S_N$ such that $d_{\mathbb{H}}(x^N_i, x^0_i)\leq B$ for some $B$ independent of $i$. This implies that for any $N$ and any $g\in S_N$, we have $d_{\mathbb{H}}(x^0_i, f'_i(g)x^0_i)\leq K$ for some $K$ independent of $i$ and $N$. Therefore $|f_i\circ \alpha_i|_S$ is bounded independent of $i$, where $S$ is a generating set of $G_0$. 

We can now apply Theorem \ref{strongversionofboundedcase} to see that \walmost surely $(f_i\circ\alpha_i)|_{G_0}$ factors through $\pi^i$. Since $\alpha_i$ is an automorphism, $f_i$ also factors through $\pi^i$. Denote the factoring map by $\tilde f_i$.

Next, we take care of the stable letters corresponding to the edges outside of $\mathbb{T}$. Let $e$ be an edge outside of $\mathbb{T}$ with edge group $E$. Let $E_1$ and $E_2$ be two subgroups of $G_0$ corresponding to $e$ and let $g_e\in G$ be the stable letter corresponding to $e$. Then \walmost surely $f_i(E_1)$ and $f_i(E_2)$ are two elliptic (or parabolic) subgroups conjugate to each other in $\Gamma_i$ by $f_i(g_e)$.  By definition of special abelian splitting, if $f_i(E_1)$ and $f_i(E_2)$ are elliptic, they are small elliptic. Hence the order of $f_i(E_1)$ and $f_i(E_2)$ goes to infinity as $i\rightarrow\infty$. As a result in both cases ($f_i(E_1)$ and $f_i(E_2)$ are elliptic or parabolic) $\tilde f_i(E_1)$ and $\tilde f_i(E_2)$ are parabolic subgroups of $\Gamma$, whose canonical projection to $\Gamma_i$ are $f_i(E_1)$ and $f_i(E_2)$, respectively. Therefore, by Lemma \ref{econp}, we know that $\tilde f_i(E_1)$ and $\tilde f_i(E_2)$ are conjugate to each other by some $\gamma^e_i$ with $\pi^i(\gamma^e_i)=f_i(g_e)$. Now we can extend $\tilde f_i$ by defining $\tilde f_i(g_e)=\gamma^e_i$.  In the same way, we extend $\tilde f_i$ to all the stable letters. 

To finish the proof, we need to extend the factoring map $\tilde f_i$ to the abelian vertex groups. Note that $\tilde f_i$ is already defined on the maximal edge groups of the abelian vertex groups of $\mathbb{G}$. Let $A$ be an abelian vertex group and $E$ its maximal edge group. Since $ker^{\omega}(f_i)=\{1\}$ and $f_i|_E$ factor through $\tilde f_i|_{E}$, we have $ker^{\omega}(\tilde f_i|_E)=\{1\}$. Recall that $E$ is shown by be cyclic. Let $g$ be a generator of $E$. Then $\tilde f_i(g)$ is a non-trivial parabolic element of $\Gamma$. Hence $\tilde f_i|_E$ is injective.  By applying Lemma \ref{abelianlemma} to $A$ and $E$ and the maps $f_i|_{A}$ and $\tilde f_i|_{E}$, one can extend $\tilde f_i|_{E}$ to $A$ if a virtually cyclic abelian group $V$ is amalgamated to the maximal parabolic subgroup $P$ of $\Gamma$ containing $\tilde f_i(E)$, i.e. $\Gamma$ is replaced by $\Gamma*_{P}V$. Note that $V$ does not depend on $i$. Repeat the above process to all abelian vertex groups of $\mathbb{G}$ yields the following:
\begin{enumerate}
\item A group $\Gamma'=\Gamma*_{\bar p}\bar V$. Here $\bar V=(V_1, ...., V_k)$, where $V_j$ is the amalgamated product of finitely many virtually cyclic abelian groups over infinite cyclic subgroups.  
\item $\pi'_i: \Gamma'\rightarrow \Gamma_i$, which is an extension of $\pi^i:\Gamma\rightarrow \Gamma_i$. 
\item $\tilde f_i:G\rightarrow \Gamma'$ such that $\pi'_i\circ\tilde f_i=f_i$. 
\end{enumerate}

By construction $\pi'_i(V_j)$ is contained in an elliptic (or parabolic) subgroup of $\Gamma_i$. Hence $\pi'_i$ factors through $\pi''_i: \Gamma*_{\bar p}\bar V'\rightarrow \Gamma_i$, where $\bar V'=(V'_1, ...., V'_k)$ with $V'_j$ being the abelianization of $V_j$.  
 Note that $V'_j$ is virtually cyclic and hence has the form $\mathbb{Z}\oplus F_j$ for some finite abelian group $F_j$. Let $K$ be the normal subgroup of $\Gamma*_{\bar p}\bar V'$ generated by $\bigcup_j ker^{\omega}(\pi''_i|_{F_j})$. Since $\bigcup_j ker^{\omega}(\pi''_i|_{F_j})$ is finite, $\pi''_i$ factors through $(\Gamma*_{\bar p}\bar V')/K$ \walmost surely.  Let $\pi_{f_i}:(\Gamma*_{\bar p}\bar V')/K\rightarrow\Gamma_i$ be that factoring map. Note that $f_i$ also factors through $\pi_{f_i}$. To finish the proof, we need to show that $(\Gamma*_{\bar p}\bar V')/K$ has the form $\Gamma*_{\bar p}\bar A$ for some $\bar A=(\mathbb{Z}\oplus(\mathbb{Z}/a_1\mathbb{Z}), \dots, \mathbb{Z}\oplus(\mathbb{Z}/a_k\mathbb{Z}))$. First note that $(\Gamma*_{\bar p}\bar V')/K$ has the form $(\Gamma*_{\bar p}\bar V'')$, where $\bar V''=(V''_1, ...., V''_k)$ and $V''_j=\mathbb{Z}\oplus (F_j/ker^{\omega}(\pi''_i|_{F_j}))$. Since $F_j/ker^{\omega}(\pi''_i|_{F_j})$ is finite, $\pi_{f_i}|_{F_j/ker^{\omega}(\pi''_i|_{F_j})}$ is injective \walmost surely. So $F_j/ker^{\omega}(\pi''_i|_{F_j})$ is cyclic \walmost surely. 
 \end{proof}

We are not ready to prove Theorem \ref{MT2}.

\begin{proof}[Proof of Theorem \ref{MT2}]
Let $L_1=G/ker^\omega(f_i)$. Then by Lemma \ref{factorthroughlimitgroup}, $f_i$ factor through the quotient map $G\rightarrow L_1$ \walmost surely. Denote the factoring map by $f_i^1:L_1\rightarrow \Gamma_i$. Then it suffices to prove the theorem for $f_i^1$. Repeating the above process produces a sequence of epimorphisms $L_1\rightarrow\cdots\rightarrow L_n\rightarrow \cdots$ of discrete ${\rm PSL}(2, \mathbb{R})$-limit groups and the corresponding sequence of maps $\{f_i^j\}_{j=1,\dots}$. By Lemma \ref{lgs}, the above epimorphism sequence stabilizes. Hence the $\omega$-kernel of $\{f_i^j\}$ is trivial for large $j$. Since it suffices to prove the theorem for $\{f_i^j\}$, we abuse notation and assume the $\omega$-kernel of $\{f_i\}$ is trivial. 

Note that if the theorem holds for all freely indecomposable factors of a Grushko decomposition of $G$, then the theorem holds for $G$. Hence, we can assume that $G$ is freely indecomposable without loss of generality. As a result, each of the non-abelian vertex groups of any special abelian splitting is freely indecomposable relative to its adjacent edge groups. Let $\mathbb{G}$ be a special abelian splitting with respect to $\{f_i\}$. Perform a good refinement (Definition \ref{goodrefine}) on $\mathbb{G}$ if necessary, we can assume that each non-abelian vertex group $N$ of $\mathbb{G}$ has a good generating set $S_N$ relative to its adjacent edge groups. 

Let $\mathcal{E}_N$ be the collection of edge groups of $N$ and let $\mathcal{A}$ be the collection of all the abelian vertex groups of $\mathbb{G}$. Let $\alpha_i\in {\rm Aut}(G, \mathcal{A})$ such that $(f_i\circ\alpha_i)|_{N}$ is short with respect to ${\rm Aut}(N, \mathcal{E}_N)$ and $S_N$ for every non-abelian vertex group $N$. To see that such $\alpha_i$ exists: Let $\alpha_i^N\in {\rm Aut}(N, \mathcal{E}_N)$ such that $f|_{N}\circ\alpha_i^N$ is short with respect to ${\rm Aut}(N, \mathcal{E}_N)$ and $S_N$. Note that $\alpha_i^N$ has a natural extension $\tilde\alpha_i^N\in {\rm Aut}(G, \mathcal{A})$ (See \cite[Definition 3.13]{R}).  Let $\alpha_i$ be the composition of all these $\tilde\alpha_i^N$. Then $\alpha_i$ is what we want. Note that $\mathbb{G}$ is a special abelian splitting with respect to $\{f_i\circ\alpha_i\}$ since $\alpha_i$ acts on the abelian vertex group of $\mathbb{G}$ by conjugation. Note that since $\alpha_i$ is automorphism, it suffices to prove the theorem for $f_i\circ\alpha_i$. We abuse notation and assume $f_i|_{N}$ is short with respect to ${\rm Aut}(N, \mathcal{E}_N)$ for all non-abelian vertex group $N$. 

Note that after shortening $f_i$, the $\omega$-kernel might become non-trivial. In that case, repeat the process of the above three paragraph. By Lemma \ref{lgs}, this process terminate after finitely many steps. So we can again assume the $\omega$-kernel of $\{f_i\}$ is trivial.

If $|f_i|_{S_N}$ is bounded for all non-abelian vertex group $N$, then Proposition \ref{vertexbounded} applies and the conclusion of the theorem follows. 

Suppose ${\rm lim}^{\omega} |f_i|_{S_N}=\infty$ for some non-abelian vertex group $N$. If Theorem \ref{GMWtree} can be applied to all non-abelian vertex group $N$ with $|f_i|_{S_N}\rightarrow \infty$. Then we apply Proposition \ref{coreproposition} to obtain a special abelian splitting $\mathbb{G}_1$ with either more edges or bigger abelian rank. Repeat the above process for $\mathbb{G}_1$ whenever possible. We get a sequence of special abelian splittings $\{\mathbb{G}_1, \dots, \}$ with strictly growing number of edges or AFF-rank. Since there are bounds on the number of edges (Corollary \ref{boundonnumofedges}) and abelian rank (Lemma \ref{boundedAFF}), there exists $\mathbb{G}_k$ to which Proposition \ref{coreproposition} can not be applied. Again, we abuse notation and denote $\mathbb{G}_k$ by $\mathbb{G}$.

By the properties of $\{f_i\}$ and $\mathbb{G}$ that we have established, Proposition \ref{coreproposition} fail to apply for one of the following two reasons: 
\begin{enumerate}
\item $|f_i|_{S_N}$ is bounded for all non-abelian vertex group $N$ of $\mathbb{G}$.
\item Theorem \ref{GMWtree} can NOT be applied to some non-abelian vertex group $N$ with $|f_i|_{S_N}\rightarrow \infty$. 
\end{enumerate}
In the first case, we are done by Proposition \ref{vertexbounded}. Now suppose we are in the second case. Note that hypothesis (1), (2) and (3) of Theorem \ref{GMWtree}  are satisfied by all non-abelian vertex groups of $\mathbb{G}$. Hence some non-abelian vertex group $N$ with $|f_i|_{S_N}\rightarrow \infty$ hypothesis (4) of Theorem \ref{GMWtree} does not hold, i.e. none of the stabilizers of non-degenerate segments in $\mc{T}$ is elliptic or parabolic, where $\mc{T}$ is the minimal $N$-invariant subtree of the asymptotic cone $\mc{X}=\mc{X}(f_i, S_N)$. Hence Proposition \ref{shortening} can be applied to $N$ and $\{f_i|_N\}$ and $S_N$. As a result, the \w kernel of $\{f_i|_N\}$ is non-trivial. Therefore the \w kernel of $\{f_i\}$ is non-trivial. Contradiction! 
\end{proof}

\section{\large ${\rm PSL}(2, \mathbb{R})$-limit group and MR-diagrams}\label{application}

With the notations defined in Definition \ref{orbifolddef} and Definition \ref{amalg}, the strong version of Theorem \ref{TMR} can be stated as follow:

\begin{theorem}\label{finite factoring set}

Let $G$ be a finitely generated group. There exist a finite collection $\mc{F}$ of finitely generated groups of the form $\pi_1(\mc{O})*_{\bar p}\bar A$, where $\mc{O}$ is an orientable hyperbolic 2-orbifold, $\bar p$ is a tuple of punctures of $\mc{O}$ and $\bar A=(\mathbb{Z}\oplus(\mathbb{Z}/a_1\mathbb{Z}), \dots, \mathbb{Z}\oplus(\mathbb{Z}/a_{|\bar p|}\mathbb{Z}))$, with the following properties: 
For each $f\in {\rm Hom}_d(G, {\rm PSL}(2, \mathbb{R}))$, there exist $\pi_1(\mc{O})*_{\bar p}\bar A\in \mc{F}$ and $\bar n \in \mathbb{N}^{|\bar p|}$. 
such that $f=\iota\circ q \circ \tilde f$, where 
\begin{enumerate}
\item $\tilde f\in {\rm Hom}(G, \pi_1(\mc{O})*_{\bar p}\bar A)$, 
\item $q:\pi_1(\mc{O})*_{\bar p}\bar A\rightarrow\pi_1(\mc{O}_{\bar n}) $ is an extension of the natural quotient map from $\pi_1(\mc{O})$ to $\pi_1(\mc{O}_{\bar n})$ and 
\item $\iota$ is a discrete faithful representation of $\pi_1(\mc{O}_{\bar n})$ into ${\rm PSL}(2, \mathbb{R})$. 

\end{enumerate}

\end{theorem}

\begin{proof}
Suppose the theorem is not true. Let $\{\Gamma^1, \Gamma^2, \dots\}$ be an enumeration of the collection of all groups of the same form as those in $\mc{F}$. Then for any $i$, there exits discrete representation $f_i$ of $G$ into ${\rm PSL}(2, \mathbb{R})$ such that $f_i$ does not factor through any of $\{\Gamma^1, \dots, \Gamma^i\}$ in the way described by the theorem. Let $\mc{O}_i$ be the orientable 2-orbifold of finite type corresponding to $\mathbb{H}/f_i(G)$. Since the rank of $f_i(G)$ is bounded by the rank of $G$, after passing to subsequence, we can assume that the underlying surfaces of $\mc{O}_i$ are the same and $\mc{O}_i$ has the same number of cone points. Hence, pass to subsequence if necessary, either $\mc{O}_i$ is isomorphic to a fixed 2-orbifold for all large $i$  or some of the cone points of $\mc{O}_i$ has order going to $\infty$ as $i$ goes to $\infty$.  In the first case, $\pi_1(\mc{O}_i)=\Gamma^m$ for some $m$. But by construction $f_i$ does not factor through $\Gamma^m$ when $i\geq m$. Contradiction.  Now we consider the second case. Let $\mc{O}$ be the orientable hyperbolic 2-orbifold  obtained from $\mc{O}_i$ by replacing all cone points whose orders go to $\infty$ with punctures.  Let $\Gamma_i$ be the fundamental group of $\mc{O}_i$ and $\Gamma$ be the fundamental group of $\mc{O}$. Note that all these 2-orbifold has hyperbolic structures. Hence by Theorem \ref{MT2}, we know that there is $\bar A=(\mathbb{Z}\oplus(\mathbb{Z}/a_1\mathbb{Z}), \dots, \mathbb{Z}\oplus(\mathbb{Z}/a_{|\bar p|}\mathbb{Z}))$ and $\Gamma*_{\bar p}\bar A$ such that $f_i$ factors through $\Gamma*_{\bar p}\bar A$. Note $\Gamma*_{\bar p}\bar A=\Gamma^m$ for some $m$. Hence by our choice of $f_i$, we know that $f_i$ does not factor through $\Gamma*_{\bar p}\bar A$ for all $i\geq m$. Contradiction. 
\end{proof}

\begin{remark}
For any cyclic subgroup $C\subset \pi_1(\mc{O})*_{\bar p}\bar A$ and any $\gamma$ not commuting with $C$, we have that $\gamma^{-1}C\gamma\cap C$ is finite. Hence $\pi_1(\mc{O})*_{\bar p}\bar A$ does not contain any Baumslag-Solitar groups. Therefore each group $\pi_1(\mc{O})*_{\bar p}\bar A$ is a hyperbolic group by Bestvina and Feighn's combination theorem (The last corollary in the introduction of \cite{BestFeighn}) as $\pi_1(\mc{O})$ and $\mathbb{Z}/a_j\mathbb{Z}$ both are.   
\end{remark}

The following corollary of Theorem \ref{finite factoring set} gives the precise description of Makanin-Razborov diagram introduced in the first section. 

\begin{corollary}\label{MR-diagram}
Let $G$ be a finitely generated group. Then there exists a finite collection $\mc{F}$ of groups and a finite directed rooted tree $T$ with root $v_0$ satisfying
\begin{enumerate}
\item Each a group in $\mc{F}$ has the form $\pi_1(\mc{O})*_{\bar p}\bar A$, where $\mc{O}$ is an orientable hyperbolic 2-orbifold, $\bar p$ is a tuple of punctures of $\mc{O}$ and $\bar A=(\mathbb{Z}\oplus(\mathbb{Z}/a_1\mathbb{Z}), \dots, \mathbb{Z}\oplus(\mathbb{Z}/a_{|\bar p|}\mathbb{Z}))$.
\item The vertex $v_0$ of $T$ is labeled by $G$. 
\item Any vertex $v\neq v_0$ of $T$ is labeled by a $\Gamma$ limit group $G_v$ for some $\Gamma\in \mc{F}$.
\item Any edge $e=[v, w]$ of $T$ is labeled by an epimorphism $\pi_e: G_v\rightarrow G_w$, where $v$ and $w$ are the initial and terminal vertices of $e$, respectively.   
\end{enumerate}
such that for any discrete representation $f: G\rightarrow {\rm PSL}(2, \mathbb{R})$ there exist a directed path $[v_0, v_1, \dots, v_l]$ from $v_0$ to some vertex $v_l$, $\pi_1(\mc{O})*_{\bar p}\bar A\in \mc{F}$ and $\bar n\in \mathbb{N}^{|\bar p|}$ such that 
\begin{equation*}
f= \iota\circ q \circ\phi\circ \pi_{l}\circ \alpha_{l-1}\circ \cdots \circ \alpha_1\circ \pi_1
\end{equation*}
where 
\begin{enumerate}
\item $\pi^i$ is the epimorphism labeling the edge $[v_i, v_{i+1}]$;
\item $\alpha_i\in {\rm Mod}(G_{v_i})$ (See Section 3 of \cite{R} for the definition.); 
\item $\phi$ is locally injective homomorphism to $\pi_1(\mc{O})*_{\bar p}\bar A$. (See Section 5.2 of \cite{R} for the definition of locally injective.)
\item $q:\pi_1(\mc{O})*_{\bar p}\bar A\rightarrow\pi_1(\mc{O}_{\bar n}) $ is an extension of the natural quotient map from $\pi_1(\mc{O})$ to $\pi_1(\mc{O}_{\bar n})$ 
\item $\iota$ is a discrete faithful representation $\pi_1(\mc{O}_{\bar n})\rightarrow {\rm PSL}(2, \mathbb{R})$.\end{enumerate}
\end{corollary}

\begin{proof}
The follows directly from Theorem \ref{finite factoring set} and \cite[Theorem 5.7]{R} since each group in $\mc{F}$ is a hyperbolic group. (Note that \cite[Theorem 5.7]{R} applies to all hyperbolic groups since they are all equationally Noetherian.)
\end{proof}

\begin{remark}
Since $\iota$ is faithful and $\phi$ is closed to being faithful, the finite diagram of epimorphisms together with the finitely many Modular groups of ${\rm PSL}(2, \mathbb{R})$-discrete limit groups uniformly encode most of the unfaithfulness of all the discrete representations from $G$ to ${\rm PSL}(2, \mathbb{R})$ besides the well understood maps $q$. The key condition that makes the Corollary \ref{MR-diagram} and Theorem \ref{finite factoring set} non-trivial is the finiteness of $\mc{F}$ and the tree $T$. 
\end{remark}

We prove Theorem \ref{finite pre finite abe} in the rest of this section. Recall that a group is {\em coherent} if all its finitely generated subgroups are finitely presented.

\begin{lemma}\label{coherent}
$A*_{F}B$ and $A*_F$ are coherent provided that $A$, $B$ are finitely generated coherent groups and $F$ is finitely generated virtually abelian. 
\end{lemma}

\begin{proof}
Let $G$ be a finitely generated subgroup of $A*_{F}B$. Let $T$ be the Bass-Serre tree of $A*_{F}B$. The action of $H$ on $T$ induces a splitting $\mathbb{G}$ of $G$. Note that the vertex groups and edge groups of $\mathbb{G}$ are (conjugates of) subgroups of $A$, $B$ or $F$. Hence they are all finitely generated since subgroups of $F$ are all finitely generated. Hence they are all finitely presented since $A$ and $B$ are coherent. Therefore $G$ is also finitely presented. 
The case of $A*F$ is similar. 
\end{proof}

The following is a summary of results from Section 6 of \cite{R}.

\begin{proposition}\label{RW}
Let $L$ be a $\Gamma$-limit group, where $\Gamma$ is a hyperbolic group. Then there exists a finitely sequence of epimorphisms 
\begin{equation*}
L=L_0\rightarrow L_1\rightarrow \cdots \rightarrow L_n
\end{equation*}
between $\Gamma$-limit groups with the following properties:
\begin{enumerate}
\item $L_n$ admits a locally injective homomorphism to $\Gamma$. 
\item $L_i$ admits a Dunwoody decomposition $\mathbb{D}_i$. 
\item Each vertex group $D^i_v$ of $\mathbb{D}_i$ admits a virtually abelian JSJ decomposition $\mathbb{A}_v^i$. 
\item The map $\pi_i:L_i\rightarrow L_{i+1}$ is injective on rigid vertex groups of $\mathbb{A}_v^i$. 
\item If all virtually abelian subgroups of $L_n$ are finitely generated, then all virtually abelian subgroups of $L_i$ are finitely generated for all $i$. 
\end{enumerate}
\end{proposition}

\begin{proposition}\label{fin gen fin pre}
Let $\Gamma$ be a hyperbolic group. Then the abelian subgroups of any $\Gamma$-Limit group are finitely generated. Suppose further that $\Gamma$ is coherent. Then $\Gamma$-Limit groups are coherent. In particular, $\Gamma$-Limit groups are finitely presented. 
\end{proposition}

\begin{proof}
Since abelian subgroups of hyperbolic groups are finitely generated, the first statement follows easily from (1) and (5) of Proposition \ref{RW}. 

Let $\pi_n: L_n\rightarrow \Gamma$ be a locally injective map. Note that all the vertex groups of $\mathbb{D}_n$ are coherent since they are mapped injectively to $\Gamma$ by $\pi_n$ and $\Gamma$ is coherent.  Hence, by Lemma \ref{coherent}, $L_n$ is coherent. Inductively, we can assume that $L_{i+1}$ is coherent. To finished the proof, it suffices to show that $L_i$ is also coherent. First note that the rigid vertex groups of $\mathbb{A}_v^i$ are all coherent as $\pi_i$ is injective on those vertex groups and $L_{i+1}$ is coherent. Since finite extension of 2-orbifold groups and virtually abelian groups are coherent and edge groups of $\mathbb{A}_v^i$ are finitely generated, by Lemma \ref{coherent}, we know that $\pi_1(\mathbb{A}_v^i)$ is coherent. Apply Lemma \ref{coherent} to $\mathbb{D}_i$, we see that $L_i$ is coherent. 
\end{proof}

\begin{theorem}\label{psllimitgroupishyperboliclimit}
Let $L$ be a ${\rm PSL}(2, \mathbb{R})$-discrete limit group. Then $L$ is isomorphic to a $\Gamma$-limit group $L'$, where $\Gamma$ has the form $\pi_1(\mc{O})*_{\bar p}\bar A$. Here $\mc{O}$ is an orientable hyperbolic 2-orbifold, $\bar p$ is a tuple of punctures of $\mc{O}$ and $\bar A=(\mathbb{Z}\oplus(\mathbb{Z}/a_1\mathbb{Z}), \dots, \mathbb{Z}\oplus(\mathbb{Z}/a_{|\bar p|}\mathbb{Z}))$.
\end{theorem}

\begin{proof}
Let a defining sequence of $L$ be $\{f_i: G\rightarrow{\rm PSL}(2, \mathbb{R})\}$, where $G$ is a finitely generated group and $f_i(G)$ is a discrete subgroup of ${\rm PSL}(2, \mathbb{R})$. Then $L=G/Ker^{\omega}(f_i)$. By Lemma \ref{factorthroughlimitgroup} \walmost surely $f_i$ factors through $L$. Let $f'_i:L\rightarrow {\rm PSL}(2, \mathbb{R})$ be the factoring map. Then $Ker^\omega(f_i')=\{1\}$. Now apply Theorem \ref{finite factoring set} to $L$, we that \walmost surely $f'_i$ factors through some $\Gamma\in \mc{F}$. Let $\{\phi_i: L\rightarrow \Gamma\}$ be the factoring maps. Then $Ker^\omega(\phi_i)\subset Ker^\omega(f_i')=\{1\}$. So $L=L/Ker^\omega(\phi_i)$ is a limit group over $\Gamma$, which has the form described by the theorem.  
\end{proof}

Theorem \ref{finite pre finite abe} is a corollary of the above theorem. 

\begin{corollary}
${\rm PSL}(2, \mathbb{R})$-discrete limit groups are finitely presented. All abelian subgroups of ${\rm PSL}(2, \mathbb{R})$-discrete limit groups are finitely generated. 
\end{corollary}

\begin{proof}
By Theorem \ref{psllimitgroupishyperboliclimit}  any ${\rm PSL}(2, \mathbb{R})$-discrete limit group is a $\Gamma$-limit group, where $\Gamma$ is as in Theorem \ref{psllimitgroupishyperboliclimit}. Note that $\Gamma$ is hyperbolic. Since $\Gamma$ is coherent, the corollary follows from Proposition \ref{fin gen fin pre}.
\end{proof}

\appendix
\section{}

In this appendix, we prove a general algebraic facts used in the paper.

\begin{restatable}{lemma}{factabelicone}\label{abelianlemma}
Let $H$ be a finitely generated abelian group and $K$ be a cyclic subgroup of $H$. Let $p: \mathbb{Z}\rightarrow \mathbb{Z}/n\mathbb{Z}$ be the natural quotient map. Suppoe $f: H\rightarrow \mathbb{Z}/n\mathbb{Z}$ and $\phi: K\rightarrow \mathbb{Z}$ are homomorphisms such that $p\circ \phi=f|_K$. Suppose $\phi$ is injective. Then there exists:
\begin{enumerate}
\item a virtually cyclic abelian group $V$ depending only on $H$, $K$ and $\phi$, but not only $f$;
\item $\tilde p: V\rightarrow \mathbb{Z}/n\mathbb{Z}$; 
\item $\iota: \mathbb{Z}\rightarrow V$;
\item $\tilde\phi: H\rightarrow V$ 
\end{enumerate}
such that the following diagram commutes. 

\begin{tikzpicture}[node distance=1.5cm, auto]
  \node (A) {$\mathbb{Z}$};
  \node (B) [below of=A] {$V$};
  \node (C) [below of=B] {$\mathbb{Z}/n\mathbb{Z}$};
  \node (D) [left of=C] {$H$};
  \node (E) [left of=D] {$K$};
  \draw [->, dashed] (A) to node {$\iota$} (B);
  \draw [->, dashed] (B) to node {$\tilde p$} (C);
  \draw [->] (D) to node [swap]{$f$} (C);
  \draw [->]  (E) to (D);
  \draw [->]  (E) to node {$\phi$} (A);
  \draw [->, dashed]  (D) to node {$\tilde \phi$} (B);
  \draw [->, bend left=70 ] (A) to node {$p$} (C); 
\end{tikzpicture}
\end{restatable}

\begin{proof}
Since $H/K$ is finitely generated, $H/K=\mathbb{Z}^{b}\oplus F$ for some finite abelian group $F$. Let $H_1$ and $H_2$ be the pre-image of $\mathbb{Z}^{b}$ and $F$ in $H$, respectively. Then $H$ is the amalgamated direct product of $H_1$ and $H_2$ over $K$. Note that $H_2$ is virtually cyclic since $K$ is cyclic and $F$ is finite. 

Let $V$ be the amalgamated direct product of $H_2$ and $\mathbb{Z}$ identifying $K$ and $\phi(K)$. Then $V$ is a still a virtually cyclic abelian group. Let $\iota: \mathbb{Z}\rightarrow V$ be the natural embedding from $\mathbb{Z}$ to the $\mathbb{Z}$ factor of $V$. Let $\tilde p: V\rightarrow \mathbb{Z}/n\mathbb{Z}$ be the map induced by $f|_{H_2}$ and $p$. 

Since $H_1/K=\mathbb{Z}^b$, we have $H_1=K\oplus\mathbb{Z}^b$. Since $\mathbb{Z}^b$ is free, there exists $\tilde f: \mathbb{Z}^b\rightarrow \mathbb{Z}$ such that $p\circ\tilde f=f|_{\mathbb{Z}^b}$. Let $\tilde \phi_1: H_1\rightarrow V$ be the map induced by the inclusion from $K$ to the first factor $H_2$ of $V$ and $\tilde f$. Let $\tilde \phi_2: H_2\rightarrow V$ be the natural inclusion. Then $\tilde \phi_1$ and $\tilde\phi_2$ agree on $K$, therefore they induce a map $\tilde \phi: H\rightarrow V$. Now it is easy to check that the group $V$ and the maps $\tilde \phi$, $\tilde p$ and $\iota$ satisfy the conclusion of the lemma. 
\end{proof}

\bibliography{temp}

\begin{thebibliography}{99}

\bibitem{Agol-liu}
I. Agol and Y. Liu, {\em Presentation length and Simon's conjecture}. J. Amer. Math. Soc., 25(1):151-187, 2012.

\bibitem{Ali}
E. Alibegovic, {\em Makanin-Razborov diagrams for limit groups}. Geom. Topol., 11:643-666, 2007.

\bibitem{BH}
M. Bridson and A. Haefliger, {\em Metric spaces of non-positive curvature}, Spriinger 1999. 

\bibitem{BMR}
G. Baumslag, A. Myasnikov, and V. Remeslennikov, {\em Algebraic geometry over groups. I. Algebraic sets and ideal theory}. J. Algebra 219 (1999), no. 1, 16-79.

\bibitem{Bestvina}
M. Bestvina, {\em Degenerations of the hyperbolic space}, Duke Math. Jour. 56 (1988), 143-161. 

\bibitem{BestFeighn}
M. Bestvina and M. Feighn, {\em A combination theorem for negatively curved groups}, J. Differential Geometry 35 (1992), 85-101. 

\bibitem{CK}
M. Casals-Ruiz and I. Kazachkov, {\em On systems of equations over free products of groups}, J. Algebra 333 (2011), 368-426. 

\bibitem{G1}
D. Groves, {\em Limit groups for relatively hyperbolic groups, I: The basic tools}, Algebr. Geom. Topol. 9 (2009), 1423-1466. 

\bibitem{G2}
D. Groves, {\em Limit groups for relatively hyperbolic groups. II: Makanin-Razborov diagrams}, Geom. Topol. 9 (2005), 2319-2358. 

\bibitem{GHL}
D. Groves, M. Hull and H. Liang, {\em Homomorphisms to 3-manifold groups}, in preparation. 

\bibitem{Guirardel}
V. Guirardel, {\em Actions of finitely generated groups on $\mathbb{R}$-trees}, Ann. Inst. Fourier, Grenoble 58, 1(2008), 159-211


\bibitem{J-M}
T. Jorgensen and A. Marden, {\em Algebraic and geometric convergence of Kleinian groups}, Math.
Scand. 66 (1990), no. 1, 47-72

\bibitem{JS}
E. Jaligot and Z. Sela, {\em Makanin-Razborov Diagrams over free products}, Illinois J. Math. 54 (2010), no. 1, 19-68. 


\bibitem{Katok}
S. Katok, \newblock {\em Fuchsian Groups}, 
Chicago Lecutures in Math, University of Chicago Press (1992).

\bibitem{liu}
Y. Liu, {\em A Jorgensen-Thurston theorem for homomorphisms}, Algebr. Geom. Topol. 12 (2012), no. 3, 1301-1311.

\bibitem{Kapovich}
M. Kapovich, {\em Hyperbolic manifolds and discrete groups}, Modern Birkhauser Classics, Birkhauser, Boston (2009)

\bibitem{KM}
O. Kharlampovich and A. Myasnikov, {\em Elementary theory of free non-abelian groups}, J.
Algebra 302 (2006) 451-552

\bibitem{R}
C. Reinfeldt and R. Weidmann, {\em Makanin-Razborov diagrams for hyperbolic groups}, preprint, 
http://www.math.uni-kiel.de/algebra/de/weidmann/research/material-research/mr2014-pdf, 2014.

\bibitem{RS1}
E. Rips and Z. Sela, {\em Structure and rigidity in hyperbolic groups}, I, Geom. Funct. Anal. 4 (1994) 337-371

\bibitem{RWZ}
A. W. Reid, S. C. Wang, and Q. Zhou. {\em Generalized Hopfian property, a min-
imal Haken manifold, and epimorphisms between 3-manifold groups}.  Acta
Math. Sin. (Engl. Ser.), 18(1):157-172, 2002.

\bibitem{Sela1}
Z. Sela, {\em Acylindrical accessibility for groups}, Invent. Math., 129(3) 527-565, 1997.

\bibitem{Sela2}
Z. Sela, {\em Diophantine geometry over groups. I. Makanin-Razborov diagrams}, Publ. Math. Inst. Hautes Etudes Sci. (2001)31-105

\bibitem{Sela3}
Z Sela, {\em Diophantine geometry over groups VII: The elementary theory of a hyperbolic group}, Proc. London Math. Soc. 99(2009), 217-273.


\bibitem{Weidmann}
R. Weidmann, {\em The Nielsen method for groups acting on trees}, Proc. London Math. Soc.  (3), 85 (1) 93-118, 2002. 


\end{thebibliography}

\end{document}